\documentclass[10pt,a4paper]{article} 
\usepackage{amsmath,amssymb,amsthm,amsfonts,amscd,euscript,verbatim, t1enc, newlfont, graphicx, cancel}
\usepackage{hyperref}
\usepackage[all]{xy}
\providecommand{\keywords}[1]{\textbf{\textit{Key words and phrases }} #1}
\providecommand{\subjclass}[1]{\textbf{\textit{2020 Mathematics Subject Classification.}} #1}

\theoremstyle{definition}
\newtheorem{theo}{Theorem}[subsection]

\newtheorem{theore}{Theorem}[section]
\newtheorem{pr}[theo]{Proposition}

 \newtheorem{lem}[theo]{Lemma}
 \newtheorem{coro}[theo]{Corollary}
  
  \newtheorem{ass}[theo]{Assumption}
\theoremstyle{remark}
\newtheorem{rema}[theo]{Remark}

\newtheorem{rrema}[theore]{Remark}
\theoremstyle{definition}
\newtheorem{defi}[theo]{Definition}
\newtheorem{defin}[theore]{Definition}

\numberwithin{equation}{subsection}

\newcommand\tu{\mathcal{T}}

\newcommand\dmu{\mathcal{D}}
\newcommand\emu{\mathcal{E}}

\newcommand\xu{\mathcal{X}}
\newcommand\ywu{\mathcal{W}}
\newcommand\yu{\mathcal{Y}}
\newcommand\tux{\mathcal{D}_p}
\newcommand\tub{\mathcal{D}^b}
\newcommand\dux{\mathcal{D}_Q}
\newcommand\tup{\mathcal{D}^u}

\newcommand\au{\underline{A}}

\newcommand\bu{\underline{B}}

\newcommand\homm{\operatorname{Hom}}
\newcommand\obj{\operatorname{Obj}}

\newcommand\id{\operatorname{id}}
\DeclareMathOperator\adfu{\operatorname{Fun}_{\z}}
\DeclareMathOperator\adfur{\operatorname{Fun}_R}

\DeclareMathOperator\imm{\operatorname{Im}}
\DeclareMathOperator\co{\operatorname{Cone}}

\DeclareMathOperator\prli{\varprojlim}
\DeclareMathOperator\inli{\varinjlim}
\DeclareMathOperator\hinli{\underrightarrow{\operatorname{hocolim}}}

\newcommand\hrt{{\underline{Ht}}}

\newcommand\spe{\operatorname{Spec}}
\newcommand\modd{\operatorname{Mod}}
\newcommand\mmodd{\operatorname{mod}}

\newcommand\q{{\mathbb{Q}}}

\newcommand\z{{\mathbb{Z}}}

\newcommand\al{\alpha}

\newcommand\ns{\{0\}}
\newcommand\ab{\operatorname{Ab}}

\newcommand\cp{\mathcal{P}}
\newcommand\perpp{{}^{\perp}}
\newcommand\opp{^{op}}

\newcommand\kmp{\mathbb{K}_m(\operatorname{Proj}}
\newcommand\lo{\mathcal{LO}}
\newcommand\ro{\mathcal{RO}}

\newcommand\llo{\le_{\mathcal{LO}}}

\begin{document}

\title{
 Producing "new" semi-orthogonal decompositions 
out of "old" ones   in arithmetic geometry}
\author{Mikhail V. Bondarko
   \thanks{ 
 The main results of the paper were  obtained under support by the Leader (Leading scientist Math) grant no. 22-7-1-13-1 and by the Russian Science Foundation grant no. 20-41-04401.}}\maketitle
\begin{abstract} 
This paper is devoted to constructing new admissible subcategories and semi-orthogonal decompositions of triangulated categories out of old ones. 
 For two 
  triangulated subcategories $\tu$ and $\tu'$ of a certain $\dmu$ and a decomposition $(\ro,\lo)$ of $\tu$ we look either for a decomposition $(\ro',\lo')$ of $\tu'$ such that there are no non-zero $\dmu$-morphisms from $\ro$ into $\ro'$ and from $\lo$ into $\lo'$, or for a decomposition  $(\ro_{\dmu},\lo_{\dmu})$ of  $\dmu$ such that $\ro_{\dmu}\cap \tu=\ro$ and $\lo_{\dmu}\cap \tu=\lo$.
 We prove some general existence statements (that also extend to semi-orthogonal decompositions 
  of arbitrary length) and apply them to various derived categories of coherent sheaves over a scheme $X$ that is proper over the spectrum of a Noetherian ring $R$. This gives a one-to-one correspondence between semi-orthogonal decompositions of $D_{perf}(X)$ and $D^b(\operatorname{coh}(X)) $; the latter extend to $D^-(\operatorname{coh}(X))$, 
 $D^+_{coh}(\operatorname{Qcoh}(X))$,   $D_{coh}(\operatorname{Qcoh}(X))$, and  $D(\operatorname{Qcoh}(X))$  under 
  very mild 
   assumptions. In particular, we obtain a vast generalization 
   of a theorem of J. Karmazyn, A. Kuznetsov, and E. Shinder. 
  
  These applications rely on recent results of Neeman that express $D^b(\operatorname{coh}(X))$ and $D^- (\operatorname{coh}(X))$ in terms of $D_{perf}(X)$. We also prove and apply a   
rather similar	 new theorem 
 that relates
   $D^+_{coh}(\operatorname{Qcoh}(X))$ and  $D_{coh}(\operatorname{Qcoh}(X))$ (these are certain modifications of the bounded below and the unbounded derived category of coherent sheaves on $X$) to homological functors $D_{perf}(X)\opp\to R-\mmodd$. 
    Moreover, we discuss  an application of this theorem to the construction of certain adjoint functors.
\end{abstract}
\subjclass{18G80, 14F08, 18A40 (Primary) 18E10, 14A15, 14G40 (Secondary)}

\keywords{Triangulated category, adjoint functor, semi-orthogonal decomposition,  admissible subcategory, coherent sheaves, perfect complexes, $t$-structure, weight structure.}

\tableofcontents
 \section{Introduction}
This paper is devoted to constructing "new" {\it semi-orthogonal decompositions} (see Definitions \ref{dsort}(1) and \ref{dmult}(2) below) of certain triangulated categories out of "old" ones; the relevance of this matter is discussed in Remark \ref{rvac}(1). 
 We prove some general existence statements and apply them to various 
triangulated subcategories of $D(\operatorname{Qcoh}(X))$ (see Definition \ref{dadm}(\ref{icoh})); we always assume  $X$ to be a scheme that is proper over the spectrum of a noetherian ring $R$. 
 These applications rely on the main results of \cite{neesat} and \cite{neetc} along with the new Theorem \ref{tnee2}. 

In this introduction we re-formulate some of the results 
 of the main body of the paper in terms of {\it admissible} subcategories (cf. Remark \ref{rdec1}(2) below). We suggest  the reader to check that these statements 
  can be obtained from their "decomposition" versions (see Theorem \ref{tone}) via straightforward applications of Proposition \ref{pdec}. 
 
We start from  some definitions.
In this paper all the subcategories we consider will be assumed to be strictly full.

\begin{defin}\label{dadm}

Let $\dmu$ be a triangulated category; assume that $\tu,\tu'$, and some $\tu_i$ are
    its (strictly 
   full) triangulated subcategories.  

\begin{enumerate}
\item\label{iadm}
 We say that 
$\tu$ is {\it left} (resp. {\it right) admissible} in $\dmu$ if the embedding $\tu\to \dmu$ admits a left (resp. right) adjoint.

$\tu$ is said to be {\it  admissible} in $\dmu$ if it is both left and right admissible in it.

\item\label{irest}
We write  
$\tu\cap{\tu'}$  for the 
subcategory   of $\dmu$ whose object class  equals $\obj \tu\cap \obj \tu'$.

Moreover, 
we 
 write  $(\tu_i)\cap{\tu'}$  for the family $(\tu_i\cap{\tu'})$. 

\item\label{imor}
Given an additive category $C$ and  $X,Y\in\obj C$  we will write $C(X,Y)$ for  the set of morphisms from $X$ to $Y$ in $C$.

Moreover,  for
$D,E\subset \obj \dmu$ we write $D\perp E$ if $\dmu(X,Y)=\ns$ for all $X\in D,\
Y\in E$.

\item\label{iperp}
We will write 
$\tu^\perp_{\tu'}$ 
for the 
 subcategory  
 of $\tu'$ whose object class is $$\{Y\in \obj \tu':\ \{X\}\perp \{Y\}\ \forall X\in \obj \tu\}.$$
Dually, we set the object class of the 
 subcategory ${}^{\perp}_{\tu'}{}\tu$ 
 to be \par\noindent
  $\{Y\in \obj \tu':\ \{Y\}\perp \{X\}\ \forall X\in \obj \tu\}$. 

Moreover, 
 we will write 
   $(\tu_i)^{\perp}_{\tu'}$ for the family 
   $(\tu_i{}^{\perp}_{\tu'})$. 

\item\label{icopr}
Assume that $\dmu$ is closed with respect to small coproducts.

 Then we will write $\tu^{\coprod}$ for the smallest 
  (strict) triangulated  subcategory 
   of $\dmu$ that is closed with respect to $\dmu$-coproducts and contains 
 $\tu$.

 Moreover, we will write $\tu^{\coprod}_{\tu'}$  (resp. $(\tu_i)^{\coprod}$ and  $({\tu_i})^{\coprod}_{\tu'}$) for the category $\tu^{\coprod}\cap {\tu'}$ (resp. for the families $(\tu_i^{\coprod})$ and  $(\tu_i^{\coprod}\cap {\tu'})$).

\item\label{imod} Throughout this paper $R$ will be a commutative unital  ring. 

We set $R-\mmodd\subset R-\modd$ to be the subcategory of finitely generated $R$-modules;  $S=\spe R$.

\item\label{icoh}
  Assume that a scheme $X$ 
  proper over $S$ 
  is fixed.  We 
 set $\tux= D_{perf}(X)\subset \tub= D^b_{coh}(\operatorname{Qcoh}(X))  \subset \dmu^-= D^-_{coh}(\operatorname{Qcoh}(X)) \subset \tup= D_{coh}(\operatorname{Qcoh}(X)) \subset \dux=D(\operatorname{Qcoh}(X))$; here $D_{perf}(X)\subset 
 \dux$ is  the 
  subcategory of perfect complexes on $X$ (cf. \cite[Tag 08CM]{stacksd}), and 
 a complex $N$ 
  (in $\dux$) belongs to $\tup$ whenever all its 
  cohomology sheaves $H^i(N)$ are coherent; it also belongs to $\tub$ (resp. $\dmu^-$) if we also have $H^i(N)=0$ for $i\gg 0$ and $i\ll 0$ (resp. for $i\gg 0$ only).  
   Moreover, we 
    consider  $\dmu^+= D^+_{coh}(\operatorname{Qcoh}(X))\subset  \tup$ that is defined similarly. We discuss these categories in Remark \ref{rk}(3) below; cf. also Remark \ref{rstacks}.


\item\label{iproj} We will say that $X$ is projective over $S=\spe R$ 
 if  $X$ is a closed subscheme of the projectivization $Y$ of a vector bundle $\mathcal{E}$ over $S$.

\item\label{irlin} All $R$-linear categories in this paper will be additive. For two $R$-linear categories $\au,\bu$ we will write $\adfur(\au,\bu)$ for the category 
 of $R$-linear functors $\au\to \bu$.
\end{enumerate}
 \end{defin}

\begin{rrema}
 Clearly, all the subcategories of $\dmu$ that we 
 describe in Definition \ref{dadm}(\ref{irest}--\ref{icopr}) 
 are triangulated; recall the strictness assumption. 

\end{rrema}

\begin{theore}\label{tadm}
Let $\xu$ be a    left (resp. right)  admissible subcategory of $\tux$ and $\ywu$ be a      left (resp. right)   admissible subcategory of $\tub$ (see Definition \ref{dadm}(\ref{icoh})).

1. Then the categories $\xu^{\perp}_{\tub}$, $\xu^{\perp}_{\dmu^-}$, and  $\xu^{\perp}_{\dux}$ are left (resp. right) admissible in $\tub$,  $\dmu^-$, and $\dux$, respectively.  

Moreover,  
   $\xu^{\coprod}_{\dux}$ is left (resp. right) admissible in $\dux$ as well, and $\xu^{\perp}_{\dux}=(\xu^{\perp}_{\tub})^{\coprod}$.


2.  Assume in addition that either  $X$ is regular of finite Krull dimension or that {\it regular alterations}\footnote{This assumption is very far from being restrictive; cf.  Remark \ref{rgeom}(1) below.} 
 exist for all 
 integral closed subschemes of $X$. 
Then    $\ywu$ equals ${\xu'}^{\perp}_{\tub}$ for some  left (resp. right)   
    admissible subcategory $\xu'$ of $\tux$.

Consequently, $\ywu^{\coprod}_{\dmu^-}$ and $\ywu^{\coprod}_{\dux}$ are  left (resp. right) admissible in  $\dmu^-$ and $\dux$, respectively. 

Moreover, the correspondence $\emu\mapsto \emu\cap \tux$ gives a one-to-one correspondence between   right admissible subcategories of $\tub$ and left admissible  subcategories of $\tux$.

3. Assume that   $X$ is projective over $S=\spe R$  (see Definition \ref{dadm}(\ref{iproj})).
 Then $\xu^{\perp}_{\tup}$ is left (resp. right) admissible in $\tup$ and $\xu^{\perp}_{\dmu^+}$  is  left (resp. right) admissible in $\dmu^+$.


4. 
Assume that $X$  satisfies the assumptions both of assertion 2 and of assertion 3.  Then 
 $\ywu^{\coprod}_{\tup}$ and  $\ywu^{\coprod}_{\dmu^+}$ are left (resp. right) admissible in $\tup$ and $\dmu^+$, respectively.
\end{theore}

\begin{rrema}\label{rk} 
1. A significant part of our 
arguments can be "axiomatized"; cf. Theorem \ref{tamain} below. 

2. The "moreover" statement in Theorem \ref{tadm}(2) vastly generalizes and extends Theorem A.1 of \cite{karkuz}. 

3. Furthermore, Corollary \ref{csub}  yields that all (right and left) admissible  subcategories provided by Theorem \ref{tadm}(1) (except $\xu^{\coprod}_{\dux}$) "restrict" to the intersections of the corresponding subcategories of $\dux$ with all "support subcategories" of $\dux$ coming from unions of closed subsets of $S=\spe R$. 

4. Recall that the obvious exact functors $D^-(\operatorname{coh}(X))\to D^-_{coh}(\operatorname{Qcoh}(X))$
and $D^b(\operatorname{coh}(X))\to D^b_{coh}(\operatorname{Qcoh}(X))$ are equivalences of categories; see \cite[Tag 0FDA]{stacksd}.

On the other hand, a similar functor $D(\operatorname{coh}(X))\to D_{coh}(\operatorname{Qcoh}(X))$ is not necessarily and equivalence; see \S3 of \cite{poshnur}. However, it {\bf is} an equivalence if $X$ is regular of finite Krull dimension; see Corollary 5.12 of ibid. 
\end{rrema}

The theorem was inspired by  "duality between weight and $t$-structure" statements that were studied by the author in several papers starting from 
 \cite{bws}; see \S\ref{swt} below for more detail. Another recent ingredients are the descriptions of some of our categories as certain categories of functors from $\tux$ and $\tub$. These are provided by Theorem 0.2 of \cite{neetc} (cf. Remark \ref{rimpr}(1)  below) together with the following theorem. 


\begin{theore}\label{tnee2}

I. The 
 restricted Yoneda 
 functor 
$\yu:\tup\to \adfur(\tux\opp,R-\mmodd)$ 
 is full.

II. Assume that $X$ is projective over $S$ (in the sense of Definition \ref{dadm}(\ref{iproj})).

1. Then 
 for an object $N$ of $\dux$ we have $\dux(M,N)\in R-\mmodd$ for all $M\in \obj \tux$ if and only if $N\in \obj \tup$  (see Definition \ref{dadm}(\ref{icoh},\ref{imod})).

Moreover, for $N\in \obj \tup$ we have $N\in \obj \dmu^-$ (resp.  $N\in \obj \dmu^+$) if and only if for any $M\in \obj \tux$ we have $\{M[i]\}\perp \{N\}$ if $i\ll 0$  (resp. $i\gg 0$). 

 
 2. Assume 
   that the ring $R$ is either countable or {\it self-injective}, that is, injective as a module over itself. Then any $R$-linear homological functor $\tux\opp\to R-\mmodd$ 
 is represented by an object of $\tup$. 
 \end{theore}

\begin{rrema}\label{rnee2}
1. Clearly, 
 Theorem \ref{tnee2}(II.1) also yields 
  a similar characterization for $\tub=\dmu^-\cap \dmu^+$.

2. This theorem substantially extends Corollary 0.5 of \cite{neesat}, where only $\tub$ and  $\dmu^-$ were considered. 
 Respectively, loc. cit. hints that it suffices to
assume that X is proper over $R$ in this theorem.

Recall also that ibid. was inspired by the question of existence of certain adjoint functors; see Remark 0.7 of ibid. 
 We prove a nice general result of this sort in Corollary \ref{cadj} below and combine it with Theorem \ref{tnee2}(1) in Remark \ref{rcadj}(1). Possibly, these statements are more "practical" than the corresponding Corollary 0.4 of ibid.
 

3. We will not apply Theorem \ref{tnee2}(II.2) in this paper (yet cf. Remark \ref{rimpr}(1); we mention adjoint functors there as well). Still 
 it is worth noting that combining it with the first part of the theorem one obtains the following (if $X$ is projective over $S$ and $R$ is either countable or self-injective):  the essential image of $\yu$ consists of all $R$-linear homological 
 functors $\tux\opp\to R-\mmodd$, and  the image of its restriction  to  $ \dmu^+$ consists of all those homological functors $H: \tux\opp\to R-\mmodd$ such that for every $M\in \obj \tux$ we have $H(M[i])=\ns$   if 
   $i\gg 0$. 
Clearly, the $ \dmu^-$ and $\tub$-versions of this observation  follow from Theorem \ref{tnee2} as well; yet recall that Corollary 0.5 of \cite{neesat} gives these statements without any extra assumptions on on $R$ (and $X$).  

4. The 
  proof of Theorem \ref{tnee2}(II.1) and of the self-injective version of part II.2 originates from the proof of \cite[Theorem A.1]{bvdb}.
\end{rrema}

We also prove some more statements of this sort; see Theorem \ref{tcond} below. 


\begin{rrema}\label{rvac}
1. Recall that semi-orthogonal decompositions of certain derived categories of (quasi)coherent sheaves are important for 
 non-commutative geometry.


 Note also that Theorem \ref{tadm}(1,2) gives more non-trivial statements in the case where $D_{perf}(X)\neq D^b(X)$, that is, if $X$ is singular; 
  cf. Remark \ref{rdual}(2) below.  
 This case is somewhat less popular than the regular one. Yet some non-trivial semi-orthogonal decompositions of  $D^b(X)$ for a possibly singular $X$ are 
 provided by Theorem 6.7 
   and Corollary 6.10 of \cite{schnurdb}. 
Moreover, semi-orthogonal decompositions in the case where $X$ is a singular surface are discussed in detail in \cite{karkuz}. 

Thus the "geometric" results of the current paper appear to be relevant.

2. An alternative  version of this paper is available as \cite{bdec}.  It is more self-contained than the current version and some of the proofs are more detailed; yet several remarks are omitted.
\end{rrema}

Let us  now describe the contents  of the paper. Some more information of this sort may be found in the beginnings of sections. 

In \S\ref{sprel} we give 
 some basic definitions (mostly) related to semi-orthogonal decompositions of triangulated categories, and prove several simple properties of these decompositions.
  
  In \S\ref{smain} we prove an abstract Theorem \ref{tamain} on the existence of certain semi-orthogonal decompositions. We use it to prove Theorem \ref{tadm} (see Theorem \ref{tone} that is formulated in the language of "binary" semi-orthogonal decompositions), to deduce the easy Corollary \ref{cadj} on the existence of adjoint functors (cf. Remark \ref{rcadj})), and to study certain "support subcategories" of $\dux$. 
   Moreover, we prove Theorem \ref{tnee2}(II.1).

In \S\ref{scomm} we finish the proof of Theorem \ref{tnee2}. 
 We  deduce its second part from a general theorem on Neeman-type approximability in triangulated categories. 
  Next we 
   study  semi-orthogonal decompositions that (may) consist of more than two components. 
 We also discuss the relation of our arguments to (adjacent) weight structures and $t$-structures; those were studied in several preceding papers of the author.

The author is deeply grateful to the referee for several important comments to the text.

\section{Preliminaries}\label{sprel}

In this section we discuss simple notions related to triangulated categories and  semi-orthogonal decompositions. 

In \S\ref{snotata}  we recall some definitions and statements 
related to triangulated categories. 

In \S\ref{ssdort} we 
 define and study ("binary") semi-orthogonal decompositions. 

\subsection{A few definitions and statements}\label{snotata}

\begin{itemize}


\item For categories $C',C$ we write $C'\subset C$ if $C'$ is a 
  subcategory of $C$;  recall that we  only consider strictly full subcategories in this paper.



\item The symbol $\tu$ below will always denote some triangulated category. The symbols 
 $\dmu$, $\lo$,  and $\ro$ (possibly, endowed with indices) 
  will  also be used  for triangulated categories only.


\item For a class  $\cp\subset  \obj\tu$ we call 
a class $\cp'\subset  \obj\tu $ 
 the {\it extension-closure} of $\cp$ if $\cp'$ is the smallest 
 class of objects of $\tu$ that contains $\cp\cup\ns$ and such that for a $\tu$-distinguished triangle $A \to C \to B \to A[1]$ we have $C\in \cp'$ whenever $A,B\in \cp'$.

\item Given a distinguished triangle $X\stackrel{f}{\to}Y\to Z$ we will write $Z=\co(f)$; recall that $Z$ is determined by $f$ up to a non-canonical isomorphism.


\item All 
 coproducts in this paper will be small.

\end{itemize}



 We will also need the following well-known definitions.

\begin{defi}\label{dcomp}
Let  $\dmu$ be a  triangulated category closed with respect to (small) coproducts. 

\begin{enumerate}

\item\label{dsmall}
 An object $M$ of $\dmu$ is said to be {\it compact} (in $\dmu$) if the functor  
 $\dmu(M,-):\dmu\to \ab$ 
 respects coproducts.

\item\label{dcgen}
We will say that a (triangulated)  subcategory $\tu$ of $\dmu$ {\it compactly generates} $\dmu$ whenever $\tu$ is essentially small, its objects are compact in $\dmu$, 
 and $\dmu=\tu^{\coprod}$ (see Definition \ref{dadm}(\ref{icopr})). 

\end{enumerate}
\end{defi}

The following  statements 
 are mostly simple and well-known. 

\begin{lem}\label{lbsn}

Let $\lo$ and $\ro$ be (strictly full) triangulated subcategories of $\tu$.
Take 
$C$ to be the class of those $M\in \obj \tu$  such that there exists a distinguished triangle $L\to M\to R\to L[1]$ with $L\in \lo$ and $R\in \ro$.

1.   If  $\lo\perp \ro$ then $C$ gives a triangulated subcategory of $\tu$ as well.

2. If  $\tu$ contains all (small) coproducts of its objects, and $\lo$ and $\ro$ are closed with respect to $\tu$-coproducts then $C$ is closed with respect to $\tu$-coproducts  as well.

3. If $\tu$ is closed with respect to coproducts then $\lo^{\perp}_{\tu}=(\lo^{\coprod})^{\perp}_{\tu}$.

4. If $\tu$  is closed with respect to coproducts, $\lo$ is essentially small and consists of compact objects, and  $\lo^{\perp}_{\tu}=\ns$, then $\lo$ compactly generates $\tu$. 
\end{lem}
\begin{proof}
 Assertions 1 and 2 easily follow from Proposition 2.1.1(1,2) of \cite{bsnew}; note (for assertion 1) that $C[1]=C$.
 
 Assertion 3 is very easy; note that for any object $N$ of $\tu$ the class ${}^\perp_{\tu} \{N\}$ 
 is closed with respect to $\tu$-coproducts. 
 
 Assertion 4 is given by Proposition 8.4.1 of \cite{neebook}.
\end{proof}

\subsection{Semi-orthogonal decompositions: basics}\label{ssdort}

 Let us give some more of our central definitions.

\begin{defi}\label{dsort}

Assume that $\tu$ and $\tu'$ are 
  triangulated subcategories of a triangulated category $\dmu$

1. Let $\lo$ and $\ro$ be 
 (strictly full) 
  triangulated subcategories of $\tu$. 

We will say that the couple $D=(\lo,\ro)$ 
 is a {\it semi-orthogonal decomposition} of $\tu$ (or just gives a decomposition of $\tu$) if $\obj \lo\perp \obj \ro$ and for any $M\in \obj \tu$ there exists a distinguished triangle \begin{equation}\label{edec} L\to M\to R\to L[1]\end{equation}
 with $L\in \lo$ and $R\in \ro$. 
 
 
 2. 
  The couple $D'=D^{\perp}_{\tu'}=(\lo^{\perp}_{\tu'},\ro^{\perp}_{\tu'})$ (see Definition \ref{dadm}(\ref{iperp})) 
 is said to be {\it $\dmu$-adjacent} 
  to $D$ 
  if  $D'$ is a decomposition of $\tu'$.

3. If $D_{\dmu}=(\lo_{\dmu},\ro_{\dmu})$ is a semi-orthogonal decomposition of $\dmu$ 
  and $D_{\dmu}{\cap}{\tu'}=(\lo_{\dmu}\cap  \tu',\ro_{\dmu}\cap  \tu')$ (see   Definition \ref{dadm}(\ref{irest}))  is a semi-orthogonal decomposition of $\tu'$ then we say that $D_{\dmu}$ {\it restricts} to $\tu'$, $D_{\dmu}{\cap}{\tu'}$ is the corresponding {\it restriction}, and $D_{\dmu}$ is an {\it extension} of  $D_{\dmu}{\cap}{\tu'}$ to $\dmu$.
 
 4. We will write $D_1\llo D_2$ if  $D_i=(\lo_i,\ro_i)$ ($i=1,2$) are  semi-orthogonal decompositions of 
  $\tu$ and $\lo_1\subset \lo_2$.
\end{defi}

\begin{rema}\label{rdec1}
 1.  Clearly, $D'$ is $\dmu$-adjacent to $D$ if and only if it is $\dmu'$-adjacent to it, where $\dmu'$ is any triangulated subcategory of $\dmu$ that contains both $\tu$ and $\tu'$.

2. Semi-orthogonal decompositions described in Definition \ref{dsort}(1) may be called the "binary" ones. We 
  postpone the more general 
   "multiple" decompositions and their properties till \S\ref{smult}. 
    The reason for this is that these more general decompositions do not help in proving anything like Theorem \ref{tadm}. 

On the other hand,  binary semi-orthogonal decompositions 
 are important for our proofs even though 
  Theorems \ref{tone} below contains just a little more information than the corresponding 
   "admissible statements" in  Theorem \ref{tadm}.

3.   We will discuss some 
    predecessors of Definition \ref{dsort}(2) in Remark \ref{radj}(1) below.

\end{rema}

\begin{pr}\label{pdec}
 1. Assume that $D=(\lo,\ro)$ 
  is a  semi-orthogonal decomposition of $\tu$. 

Then $\lo^\perp_{\tu}= \ro $, $\lo={}^\perp_{\tu}  \ro$, and 
 there exists  an exact right adjoint $R_D$ to the embedding $\lo\to \tu$ and a left adjoint $L_D$ to the embedding $\ro\to\tu$.

Moreover,  the triangle 
(\ref{edec})  is functorially 
 determined by $M$, and the arrows  $L\to M\to R$ in it come from the  transformations corresponding to the aforementioned 
  adjunctions.

 2.  Respectively, the correspondence  $D\mapsto \ro$ (resp. $D\mapsto \lo$)  gives a bijection of between the class of  semi-orthogonal decompositions of $\tu$ and that of    left (resp. right) admissible subcategories of $\tu$; see Definition \ref{dadm}(\ref{iadm}). 

  \end{pr}
\begin{proof}

   These are  well-known statements.  They are mostly contained in Lemmata 2.5 and 2.3 of \cite{kuzbch}; yet invoke   Proposition 1.3.3. of \cite{bbd} (along with its proof; the 
    relation of $t$-structures mentioned loc. cit. to semi-orthogonal decompositions is discussed in \S\ref{swt} below) 
    for the calculation of the arrows in (\ref{edec}).
     \end{proof} 

Let us prove 
 some more properties of 
  our notions.

\begin{pr}\label{port}
Assume  $\tu,\tu'\subset \dmu$,  and  $D$ (resp. $D'$) 
  is a  semi-orthogonal decomposition of $\tu$ (resp. $\tu'$). 

I. If $D'=D^{\perp}_{\tu'}$ (see 
Definition \ref{dadm}(\ref{iperp})) and $\dmu$ is $R$-linear 
 (see Definition \ref{dadm}(\ref{imod}) then the following 
  bi-functors $\tu\opp\times \tu'\to R-\modd$ are canonically isomorphic: $\dmu(L_D(-),-)\cong \dmu(-,R_{D'}(-))$ and    $\dmu(R_D(-),-)\cong \dmu(-,L_{D'}(-))$. 

II. Assume that $D_1'=D^{\perp}_{1,\tu'}$ and $D_2'=D^{\perp}_{2,\tu'}$, where $D_i=(\lo_i,\ro_i)$ ($i=1,2)$) are semi-orthogonal decompositions of  $\tu$. 

1.  Then $D_1\llo D_2$ if and only if $\ro_2\subset \ro_1$. Moreover, if these conditions are fulfilled then  $D'_2\llo D'_1$. 

2. Suppose that $\tu\subset \tu'$. Then all the conditions in assertion II.1 are equivalent. 
\end{pr}
\begin{proof}
I. This statement easily follows from Proposition 2.5.4(1) of \cite{bger}; see 
 Remark \ref{radj}(2) below for more detail. 

II.1.  Obvious from our definitions along with Proposition \ref{pdec}(1).

2. We assume $D'_2\llo D'_1$; it suffices to prove that 
 $\ro_2\subset \ro_1$. 
Now, $\lo_i'=(\lo_i)^{\perp}_{\dmu}\cap \tu'$ and  $\ro_i=(\lo_i)^{\perp}_{\dmu}\cap  \tu$  (for $i=1,2$; see  Proposition \ref{pdec}(1) once again); hence $\ro_i=\lo_i'\cap  \tu$. Thus $\ro_2\subset \ro_1$ indeed.
  \end{proof}

\begin{rema}\label{rort}
 In all the examples that we consider in this paper the decompositions mentioned in Proposition \ref{port}(I)  extend to two adjacent decomposition of a ("big enough") category $\dmu$;  see  Theorem \ref{tamain}(I.2) below (and Remark \ref{rdec1}(1)). Now, if $\tu=\tu'=\dmu$ then for $M,N\in \obj \dmu$ we have bi-functorial isomorphisms $\dmu(L_D(M),N)\cong \dmu(L_D(M),R_{D'}(N))\cong \dmu(M,R_{D'}(N))$  that come from the corresponding adjunctions.

\end{rema}

\section{Main results}\label{smain}

In \S\ref{sabs} we prove our main abstract Theorem \ref{tamain} on the existence of certain $\dmu$-adjacent semi-orthogonal decompositions. We also deduce a simple 
 Corollary \ref{cadj} on the existence of adjoint functors.

In \S\ref{sgeom} we apply our general results to semi-orthogonal decompositions of various subcategories of $D(\operatorname{Qcoh}(X))$ (where $X$ is proper over the spectrum of a ring $R$); this 
 yields 
  a "geometric" Theorem \ref{tone} on semi-orthogonal decompositions in $D_{perf}(X)$, $D^b_{coh}(X)$, $D^-_{coh}(X)$, $D^+_{coh}(\operatorname{Qcoh}(X))$, and $D_{coh}(\operatorname{Qcoh}(X))$. 
    Moreover, Corollary \ref{csub} says that these ($D(\operatorname{Qcoh}(X))$-adjacent) decompositions can be restricted to certain "support subcategories" of the corresponding categories (that correspond to unions of closed subsets of $S=\spe R$).

In \S\ref{sunb} we prove Theorem \ref{tnee2}(II.1). We also apply Grothendieck duality arguments to establish the "regular" case of Theorems \ref{tadm} and \ref{tone}
 and relate semi-orthogonal decompositions to duality; see  Corollary \ref{cdual}.

\subsection{Abstract decomposition statements}\label{sabs}


Now we study certain decompositions coming from semi-orthogonal decompositions of categories of compact objects. We will use 
 much 
  of  Definitions \ref{dadm} and \ref{dcomp}.

\begin{theo}\label{tamain}

Assume 
 that $\dmu=\tu^{\coprod}$, 
    where $\tu\subset \dmu$ is a triangulated subcategory  whose objects are $\dmu$-compact, 
    and $D=(\lo,\ro)$ is a semi-orthogonal decomposition of $\tu$. 

I.1. 
Then the couple $D^{\coprod}$ 
  is a semi-orthogonal decomposition of $\dmu$. 

2. Assume in addition that $\tu$ is essentially small (respectively, $\dmu$ is compactly generated by it). Then  $D^{\perp}_{\dmu}$ is a semi-orthogonal decomposition  of $\dmu$ as well. Moreover,  $D^{\perp}_{\dmu}$  is also  $\dmu$-adjacent to $D^{\coprod}$, and $\lo^{\perp}_{\dmu}=\ro^{\coprod}$.

3. Assume that $\tu_0$ is a subcategory of $\dmu$ such that $D^{\perp}_{\dmu}$ restricts to a semi-orthogonal decomposition 
$D_0$ on it (see Definition \ref{dsort}(3)). 

Then $D^{\perp}_{\dmu}$ restricts to the category $\tu_{0}^{\coprod}$ 
  as well, and this restriction equals  $D_0^{\coprod}$.  


II. Assume that $R$ is a commutative unitial ring, 
 $\dmu$ is $R$-linear, and $\au$ is an exact abelian subcategory of the category $R-\modd^{\z}$ of $\z$-graded $R$-modules that is stable with respect to obvious shifts on this category (that is, $M=\bigoplus M^i$ 
  belongs to $\au$ if and only if the module $M[1]=\bigoplus M^{i+1}$ does). 
Take the following (full) subcategory $\tu_{\au}$ of $\dmu$: 
 $N\in \obj \dmu$ is an object of $\tu_{\au}$ whenever 
 for any $M\in \obj \tu$ the graded module $S_N(M)=\bigoplus_j\dmu(M[-j],N)$  belongs to $\au$. 

1. Then $\tu_{\au}$ is triangulated and the functors $L_{D^{\perp}_{\dmu}}$ and $R_{D^{\perp}_{\dmu}}$ corresponding to $D^{\perp}_{\dmu}$ (see Proposition \ref{pdec}(1) and assertion I.2) send $\tu_{\au}$ into itself.

2. Consequently, 
the couple $D^{\perp}_{\tu_{\au}}$ 
  is a decomposition of $\tu_{\au}$. 

\end{theo}
\begin{proof}
I.1. Since objects of $\lo$ are compact in $\dmu$,  the class $\lo^\perp_{\dmu}$ is closed with respect to coproducts. Since it contains $\ro$, $\lo\perp \ro^{\coprod}$. 
Applying Lemma \ref{lbsn}(3) we obtain
 that  $\lo^{\coprod}\perp \ro^{\coprod}$.

It remains to prove the existence of decompositions of the type (\ref{edec}). Take the  set $E$ of those $M\in \obj \dmu$ such that there exists a distinguished triangle $L\to M\to R\to M[1]$ with $L\in \lo^{\coprod}$ and $R\in \ro^{\coprod}$. $E$ clearly contains $\obj \tu$ and applying Lemma \ref{lbsn}(1,2) we obtain 
 $E=\obj \dmu$. 

2. Corollary 2.4 of \cite{nisao} easily implies that 
$D^{\perp}_{\dmu}$ is a semi-orthogonal decomposition of $\dmu$ indeed. 
Is is obviously $\dmu$-adjacent both to $D$ and to 
$D^{\coprod}$.  Lastly, $\lo^{\perp}_{\dmu}=\ro^{\coprod}$ by Proposition \ref{pdec}(1) (applied to 
 $D^{\coprod}$). 

3. Since all objects of $\lo$ and $\ro$ are compact, 
 both $\lo_{D^{\perp}_{\dmu}}$ and $\ro_{D^{\perp}_{\dmu}}$  are closed with respect to $\dmu$-coproducts. Hence if $D_0=(\lo_0,\ro_0)$ then 
  $\lo_{0}^{\coprod}\perp \ro_{0}^{\coprod}$. 
  
  Consequently, it suffices to verify  that the class $C_{0}$ of extensions of elements of $\ro_{0}^{\coprod}$ by those of $\lo_{0}^{\coprod}$ coincides with $\obj \tu_0^{\coprod}$; see Proposition \ref{pdec}(1). The latter statement easily follows from Lemma \ref{lbsn}(1,2) similarly to the proof of assertion I.1. 
 
 II.1.  Since the functor $\dmu(-,N)$ ($N\in \obj \dmu$) sends $\tu$-distinguished triangles  into long exact sequences (one may also say that $S_N$ sends distinguished triangles into triangles of a certain sort), 
 $\tu_{\au}$ is 
   triangulated. 
  Next, Proposition \ref{port}(I) implies that   
 for any $M\in \obj \tu$ 
 we have $S_{L_{D^{\perp}_{\dmu}}(N)}(M)\cong S_N(R_D(M))$ and  $S_{R_{D^{\perp}_{\dmu}}(N)}(M)\cong S_N(L_D(M))$.
 Since the objects $R_D(M)$ and $L_D(M)$ belong to $\tu$, we obtain that $L_{D^{\perp}_{\dmu}}(N)$ and $R_{D^{\perp}_{\dmu}}(N)$ belong to $\tu_{\au}$ whenever $N$ does.

2. We should check 
  that any object $M$ of $\tu_{\au}$ possesses a ${D^{\perp}_{\dmu}}$-decomposition   (\ref{edec}) inside $\tu_{\au}$. This statement 
   follows from assertion II.1 
    according to Proposition \ref{pdec}(1). 

  \end{proof}

\begin{rema}\label{rlin}
1. Below we will apply our theorem in the following setting only: $R$ is a (commutative unital) Noetherian ring and $\au$ consists of those modules whose components are finitely generated and satisfy some boundedness condition; see Definition \ref{dbound}.

  Note however that one can take 
   finitely presented modules over a coherent ring instead; see Definition 2.1,  Theorem 2.4, Corollary 2.7, and Lemma 2.8 of \cite{swancoh}. 
   In particular, it appears that the arguments below that we  use for the proof of Theorem \ref{tnee2}(II.1)  generalize to
    this setting without much difficulty.  

2. Another possible generalization here is to fix an infinite cardinal $\aleph$ and take 
 $\au$ that consists of those $M=\bigoplus M^i$ such that each $M^i$ has less than $\aleph$  generators over $R$. This gives a certain "smallness" filtration on $\dmu$, which is clearly exhaustive (since $\tu$ is essentially small). 
 
 We will also 
  consider certain type of  $\au$ related to support sets  $T\subset S=\spe R$ in 
   Corollary \ref{csub} below.
 
 3. It is also easily seen that for any $\aleph$ as above the semi-orthogonal decomposition $D^{\coprod}$ restricts to the smallest triangulated subcategory of $\dmu$ that it closed with respect to coproducts of less than $\aleph$ objects and contains $\tu$. Yet the corresponding 
  decompositions seem to be less interesting in the "geometric" setting that we will consider below. Moreover, the author does not know of any arguments that would allow to combine these cardinality restrictions with any bounds on the degree. 
 

 4. In our theorem we send 
  an object $N$ of $\dmu$ into 
   the class  
   $\{S_N(M)\}$ of graded $R$-modules (where $M$ runs through objects of $\tu$). Clearly, it is possible to get "more information" from the functor represented by $N$, and this may  give descriptions of a larger class of triangulated subcategories of $\dmu$. In particular, one may 
  look at $\infty$-enhancements. This can possibly be useful for the study semi-orthogonal decompositions corresponding to certain stacks; cf. Theorem  3.0.2 of \cite{benzvi}.   Note also that differential graded enhancements were used in the proof of \cite[Theorem A.1]{karkuz}; cf. Remark \ref{rk}(2).  
   Furthermore, one may combine this approach with the "$\dmu$-free one" that is discussed in Remark \ref{rimpr}(1) below.
 
 5. Assertion I.1 is possibly well-known. Note also that in the "geometric" setting that we will study below this statement essentially coincides with  Proposition 4.2 of \cite{kuzbch}. 
\end{rema}

Theorem \ref{tamain} easily yields the existence of certain adjoint functors.

\begin{coro}\label{cadj}
Adopt the assumptions of Theorem \ref{tamain}(II). Assume that $F: \dmu\to \dmu'$ is an exact functor that respects coproducts and suppose that $\tu \subset \tu_{\au}$. 

Let $T$ be a 
subcategory of $\dmu'$ that contains the essential image $F(\tu_{\au})$. For any objects $M$ of $\dmu$ and  $N$ of $\dmu'$ 
 endow the abelian group $\dmu'(F(M),N)$ with the structure of $R$-module as follows: define the multiplication by $r\in R$ on  $\dmu'(F(M),N)$ by means of composing its elements with $F(r\id_{M})$. 

Then the restriction of $F$ to a  functor $F_T:\tu_{\au}\to T$ possesses a right adjoint if and only 
if for any 
 $M\in \obj \tu$ and $N\in \obj T$ the graded module $$
 S'_N(F(M))=\bigoplus_{i\in \z}\dmu'(F(M[-i]),N)$$ 
   belongs to $\au$.
 
 Moreover, this adjoint is exact if $T$ is a triangulated subcategory of $\dmu'$.
\end{coro}
\begin{proof}
If $F_T$ possesses an adjoint functor $G_T$ then for any $M\in \obj \tu_{\au}$ and $N\in \obj T$ we have $
S'_N(F(M)) \cong \bigoplus_{i\in \z} S_{G_T(N)}(M)$  (cf. Theorem \ref{tamain}(II.1)), 
 and this isomorphism is clearly an isomorphism of $R$-modules.  Since $\tu\subset  \tu_{\au}$, we obtain that  
  the graded module $S'_N(F(M))$ belongs to $\au$ whenever $M\in \obj \tu$.

Let us prove the converse implication. Since $\dmu$ is compactly generated, the functor $F$ is well known to possess an exact right adjoint  $G$; see Theorems  8.3.3 and 8.4.4 and Lemma 5.3.6 of \cite{neebook}. Thus it suffices to verify that $G$ sends $T$ inside $\tu_{\au}$. Now, for any $N\in \obj T$ if $M$ belongs to $\tu$ then   
  $S_{G_T(N)}(M)\cong S'_N(F(M))$; hence $G(N)$ belongs to $ \tu_{\au}$ indeed. 
\end{proof}

We also describe an argument that can be used to restrict semi-orthogonal decompositions of Theorem \ref{tamain}(II)  to "large enough" subcategories of  $\tu_{\au}$. We will not apply it in this paper.

 \begin{pr}\label{ptu}
 
 Adopt the assumptions and notation of Theorem \ref{tamain}(II).
  Moreover, suppose that $\tu'$ is a triangulated subcategory of $\tu_{\au}$ such that the restricted Yoneda functor $\yu_{\tu'}:\tu'\to \adfur(\tu,R-\modd)$  (see Definition \ref{dadm}(\ref{irlin}))  that sends $N\in \obj \tu'$ into the restriction of $\dmu(-,N)$ to $\tu$ satisfies the following conditions: it is full and its essential image coincides with the image of the (similarly defined) restricted Yoneda functor $\yu_{\tu_{\au}}: \tu_{\au}\to \adfur(\tu,R-\modd)$.

Then $D^{\perp}_{\tu'}$ is a decomposition of $\tu'$.  
\end{pr}
 \begin{proof}
 For an object $N$ of $\tu'\subset \tu_{\au}$ we take its  $D^{\perp}_{\tu_{\au}}$-decomposition triangle (see (\ref{edec}) and Theorem \ref{tamain}(II.2)): $L\stackrel{h}{\to}N\to R\to L[1]$.
 
 According to our assumptions, we can choose $L'\in \obj \tu'$ and $h'\in \tu'(L',N)$ such that $\yu_{\tu_{\au}}(h')\cong \yu_{\tu_{\au}}(h)$ (in the category of objects over $\yu_{\tu_{\au}}(N)$). Clearly, $L'\in \lo^{\perp}_{\tu'}$.
 It remains to verify that $\co(h')\in \ro^\perp_{\dmu}$.
 
 If $M\in \ro$ then $M[i]\perp R$ for any $i\in \z$; hence the homomorphisms $\dmu(M[i],h)\cong \dmu(M[i],h')$ are 
  bijective. Looking at the exact sequence $$\dmu(M,L')\to \dmu(M,N)\to \dmu(M,\co(h'))\to  \dmu(M[1],L')\to \dmu(M[1],N)$$ we obtain that $M\perp\co(h')$. Thus $h$ can be completed 
 to a $D^{\perp}_{\tu'}$-decomposition triangle for $N$.
 \end{proof}

\begin{rema}\label{rptu}
This statement may be combined with Corollary 0.5 of \cite{neesat} to obtain a "substitute" for Theorem \ref{tison} below (cf. Remark \ref{rison}) that would be sufficient to establish Theorem \ref{tone}(I.1).

However, the author does not have any "interesting" examples for this proposition, that is,  with $\tu'$ 
  distinct from $\tu_{\au}$ (recall that $\tu'$ is a strict subcategory). 
 \end{rema}

\subsection{Main geometric  applications}\label{sgeom}
Till the end of the paper we will always assume that the following condition is fulfilled.

\begin{ass}\label{assm} $R$ is a commutative unital Noetherian ring and $X$ is a scheme proper over $S=\spe R$.\end{ass}

In some of the statements we will also need the following "very common" condition on $X$.  

\begin{ass}\label{assa}
 {\it Regular alterations} (see  Remark \ref{rgeom}(1) below) exist for all 
 integral closed subschemes of $X$.\end{ass}

\begin{defi}\label{dbound}
We will write $\au^-$ (resp. $\au^+$) 
 for the following 
   subcategories of  $\au^u=R-\mmodd^{\z}$ (see Definition \ref{dadm}(\ref{imod})): 
   $M=\bigoplus M^i\in \obj \au^-$ (resp. $\au^+$)  whenever $M^i=\ns$ for $i\gg 0$ (resp. for 
   $i\ll 0$). 
   
    Moreover, $\au^b$ is the 
     subcategory corresponding to $\obj \au^-\cap \obj \au^+$
\end{defi}

Let us describe 
  some examples for the assumptions of 
  Theorem \ref{tamain}. We should recollect some statements from \cite{stacksd} that allow us to apply the results of ibid. to various categories of coherent sheaves.

\begin{rema}\label{rstacks}
The "main" derived categories of  \cite{stacksd} are the derived categories of $\mathcal{O}_X$-modules. However,
$D(\mathcal{O}_X)$ contains a full triangulated subcategory $D_{\operatorname{Qcoh}}(\mathcal{O}_X)$ consisting of those complexes whose cohomology is quasi-coherent. 
Now, if $X$ is Noetherian then the obvious functor $\dux=D({\operatorname{Qcoh}}(X))\to D_{\operatorname{Qcoh}}(\mathcal{O}_X)$ is an equivalence; see  \cite[Tag 09T4]{stacksd}. It clearly follows that $D_{\operatorname{coh}}({\operatorname{Qcoh}}(X))\cong D_{\operatorname{coh}} (\mathcal{O}_X)$ (cf.  Definition \ref{dadm}(\ref{icoh}) or  \cite[Tag 06UP]{stacksd} for categorical notation of this sort).

Moreover, the direct and inverse image functors (that is, $f_*: D(\mathcal{O}_X) \leftrightarrows D(\mathcal{O}_Y):f^* $ for a quasi-separated and quasi-compact morphism $f:X\to Y$ of schemes) and tensor products respect the subcategories $D_{\operatorname{Qcoh}}(\mathcal{O}_{-})$ of $D(\mathcal{O}_{-})$; see  \cite[Tags 08DW, 08D5 08DX]{stacksd}.
These observations allow us to apply results of ibid. to the categories $D(\operatorname{Qcoh}(-))$ and their subcategories mentioned in Definition \ref{dadm}(\ref{icoh}) instead of $D_{\operatorname{Qcoh}}(\mathcal{O}_{-})\subset D(\mathcal{O}_{-})$ and the corresponding triangulated subcategories that are defined in terms of cohomology of complexes of sheaves of modules (similarly to  Definition \ref{dadm}(\ref{icoh})). 
\end{rema}

\begin{theo}\label{tcond}
I. 
 Take  $\tu=\tux$ and $\dmu=\dux$ (see Definition \ref{dadm}(\ref{icoh})).

1. Then  $\dmu$ is compactly generated by  $\tu$.

2. Let the symbol s be equal to $u,+,-$, or $b$. Then we have $\dmu^s\subset \tu_{\au^s}$.

3. Moreover, 
 this inclusion is an equality  if either $X$ is projective over $S$ 
   or if $s\in \{b,-\}$. 

II. Take $\dmu$ 
to be  
the mock homotopy category $\kmp X)$ of projectives over $X$ as defined in   \cite[Definition 3.3]{mur}, and $\tu$ to be the essential  image of  
$\tub{}\opp$ under the 
 functor $(-)^{\circ} U_{\lambda}$. 

1. Then  $\dmu$ is compactly generated by  $\tu$ and the restriction of $(-)^{\circ} U_{\lambda}$ to 
$\tub$ is a full embedding. 

2. If
 $X$ satisfies Assumption \ref{assa}
 then the corresponding category $\tu_{\au^b}$ 
   equals  the essential image 
    $(-)^{\circ} U_{\lambda}(\tux\opp)$. 
\end{theo}
\begin{proof}
I.1. Since $R$ is noetherian, $X$ is a noetherian separated scheme; thus the compact generation statement is well-known (see 
 \cite[Tags 09M1, 09IS]{stacksd} along with Remark \ref{rstacks} (or Theorem 3.1.1 of \cite{bvdb}) 
  and  Lemma \ref{lbsn}(4). 

2. This statement is easy and probably well-known; our argument in the proof Theorem \ref{tnee2}(II.1) (in \S\ref{sunb} below) actually establishes it in this generality as well.

 3. In the case where $p$ is projective the assertion is just a re-formulation of Theorem \ref{tnee2}(II.1).
 
 In the cases 
 $s= b$ and $s=-$ 
 Corollary 0.5 of \cite{neesat} implies  
  the following: for any $N\in \obj \tu_{\au^s}$ there exists $N'\in \obj \dmu^s$ such that $\yu(N)\cong \yu(N')$; here we use the notation of Theorem \ref{tnee2}(I). To prove that $N\cong N'$ one can either apply 
  some more results of ibid. 
  (see Remark \ref{rison}(2) below) 
   or use Theorem  
   \ref{tison}(I.1,III).



II.1. These statements are given by  
Theorems 4.10 and 
7.4 of \cite{mur};  
 see also Proposition 7.2 of ibid. for the notation.

2. 
If $X$ satisfies Assumption \ref{assa} then Theorem 0.2 of \cite{neetc} says that the objects of 
$\tux\opp\subset \tub{}\opp$ represent all homological functors 
$F:
\tub\to R-\modd$ such that $\bigoplus_{i\in \z}F(M[i])$ is a finitely generated $R$-module for any $M\in \obj \tub$. 
 It follows that for any $N\in \obj 
 \tu_{\au^b}$ there exists $N'\in  (-)^{\circ} U_{\lambda}(\tux\opp)$ such that the restrictions of the functors $\dmu(-,N)$ and $\dmu(-,N')$ to $\tu$ are isomorphic. Since $N'$  belongs to 
  $\tu$,  $\id_{N'}$ yields a canonical morphism $f:N'\to N$, and 
   we have $\obj \tu\perp \co(f)$. It easily follows that $\co(f)=0$; see 
Lemma \ref{lbsn}(3).    Hence    $\tu_{\au^b}$  equals the essential image $(-)^{\circ} U_{\lambda}(\tux\opp)$ indeed.
\end{proof}

\begin{rema}\label{rcadj}
1. Clearly, one can combine Corollary \ref{cadj} with Theorem \ref{tcond}(I.3) to obtain an if and only if 
 criterion for the existence of a right adjoint to the corresponding restriction 
 $F_T:\dmu^s\to T$, where  $F:\dux\to \dmu$ is an exact  functor that respects coproducts, $F(\dmu^s)\subset T$, and   $(X,s)$ is any couple that satisfies the assumptions of Theorem \ref{tcond}(I.3).
 
 2. One can also construct a category $\dmu$ that satisfies the assumptions of Theorem \ref{tcond}(II) using certain abstract nonsense; see Corollary 3.7 of \cite{dgk} and Proposition 4.2.5 of \cite{bdec}.
  However, the author conjectures that all possible choices of $\dmu$ that satisfy the conditions of our theorem 
   are equivalent.
\end{rema}

 Now we 
  pass  to semi-orthogonal decompositions; see Definitions \ref{dsort} and \ref{dadm}.

\begin{theo}\label{tone}

I.1. 
 For any semi-orthogonal decomposition $D$ of $\tux$ the couples $D^{\perp}_{\tub}$, $D^{\perp}_{\dmu^-}$, and $D^{\perp}_{\dux}$ give semi-orthogonal decompositions of $\tub$, $\dmu^-$, and $\dux$, respectively. 
 
Moreover, $D^{\perp}_{\dux}=(D^{\perp}_{\tub})^{\coprod}$; 
 consequently,  
 $D^{\perp}_{\dmu^-}= (D^{\perp}_{\tub})^{\coprod}_{\dmu^-}$. 

Furthermore, if $X$ is projective over $S$ (see Definition \ref{dadm}(\ref{iproj})) then 
$D^{\perp}_{\tup}$ 
 is a  decomposition  of $\tup$ and $D^{\perp}_{\dmu^+}$ is a decomposition  of $\dmu^+$. 

2. The maps $D\mapsto D^{\perp}_{\tub}$, $D\mapsto D^{\perp}_{\dmu^-}$, $D\mapsto D^{\perp}_{\dux}$ are injective, and the following assumptions on semi-orthogonal decompositions $D_1$ and $D_2$ of $\tux$ are equivalent:

(a) $D_1\llo D_2$;

(b) $D_2{}^{\perp}_{\tub}\llo (D_1)^{\perp}_{\tub}$;

(c) $D_2{}^{\perp}_{\dmu^-}\llo (D_1)^{\perp}_{\dmu^-}$;

(d) $(D_2)^{\perp}_{\dux}\llo (D_1)^{\perp}_{\dux}$.

II. Suppose in addition that  $X$ is either regular of finite Krull dimension or satisfies Assumption \ref{assa}. 
  
 1.  Then the 
  map $D\mapsto  D^{\perp}_{\tub}$ 
gives all  semi-orthogonal decompositions of  $\tub$, and the inverse correspondence  is of the form $E\mapsto {}^\perp_{\tup} E$.  

Consequently, 
the couple $E^{\coprod}$ gives a semi-orthogonal decomposition of $\dux$ that coincides with $ (\perpp E)^{\perp}_{\dux}$, and this decomposition restricts to the
 semi-orthogonal decomposition $(\perpp E)^{\perp}_{\dmu^-}$ of $\dmu^-$.

 Moreover, if $X$ is projective over $S$ then $E^{\coprod}$  restricts to $\tup$ and $\dmu^+$.
 
2. 
 The map $\emu\mapsto \emu\cap \tux$ gives a one-to-one correspondence between   right admissible subcategories of $\tub$ and left admissible  subcategories of $\tux$.

\end{theo}
\begin{proof}
I.1.  Combining Theorem \ref{tamain}(I,II.2)  with 
  Theorem \ref{tcond} 
    we immediately obtain that $D^{\perp}_{\tub}$, $D^{\perp}_{\dmu^-}$, and $D^{\perp}_{\dux}$ give semi-orthogonal decompositions of the corresponding categories, indeed. 
 This is also true for $D^{\perp}_{\tup}$  and $D^{\perp}_{\dmu^+}$ whenever $X$ is projective over $S$.
 
 Next, recall that $\tux\subset \tub\subset 
  \dux$. Hence $\tux^{\coprod}=\tub{}^{\coprod}=\dux$ (see Theorem \ref{tcond}(I.1)); 
   thus $D^{\perp}_{\dux}=(D^{\perp}_{\tub})^{\coprod}$  by Theorem \ref{tamain}(I.3). 
  It clearly follows that  $D^{\perp}_{\dmu^-}= (D^{\perp}_{\tub})^{\coprod}_{\dmu^-}$.

 
 2. Immediate from Proposition \ref{port}(II).

II.1. We claim that for any semi-orthogonal decomposition $E$ of $\tub$ 
 the couple $\perpp E=({}^{\perp}_{\tux}\lo_E,{}^{\perp}_{\tux}\ro_E)$ is a decomposition of  $\tux$. 

In the case where $X$ satisfies Assumption \ref{assa} this statement follows from Theorem \ref{tamain}(I,II.2) combined with Theorem \ref{tcond}(II). 
If $X$ is regular of finite Krull dimension then the statement is given by  Corollary \ref{cdual}(I) below.

  Next,  compare the decomposition  $(\perpp E)^{\perp}_{\tub}$ (provided by assertion I.1) with $E$ using
  Proposition \ref{port}(II.1); we obtain $(\perpp E)^{\perp}_{\tub}=E$. 
   Applying assertion I.1 
 we deduce all the remaining statements in our assertion. 
 
2. Combining  
 assertion II.1 with Proposition \ref{pdec} 
  we obtain that the correspondence $\emu\mapsto \emu^{\perp}_{\tux}$ (resp. $\emu\mapsto \emu^{\perp}_{\tub}$) gives a bijection between the class of right admissible subcategories of $\tux$ and the class of   left  admissible subcategories of $\tux$ (resp. of right admissible subcategories of $\tub$). We conclude the proof by noting that $\emu^{\perp}_{\tux}=\emu^{\perp}_{\tub}\cap \tux$.
 \end{proof}

\begin{rema}\label{rgeom} 

1.  Recall that alterations were introduced in \cite{dej}; regular alterations generalize Hironaka's resolutions of singularities. Being more precise, a  regular alteration for a scheme $Z$ is a proper surjective morphism $Y\to Z$ that is generically finite 
and such that $Y$ is regular and finite dimensional. 

 Since 
  resolutions of singularities exist for arbitrary 
   quasi-excellent $\spe\q$-schemes according to  Theorem 1.1 of \cite{temr}, part II of our proposition can be applied whenever $R$ is a quasi-excellent noetherian $\q$-algebra. Moreover, 
 Assumption \ref{assm} is fulfilled whenever $X$ is of finite type over a scheme $B$ that is quasi-excellent of dimension at most $3$; see Theorem 1.2.5 of \cite{tema}.


2. 
It appears to be no easy way to prove that semi-orthogonal decompositions of  $\tub$   extend to $\dmu^-\subset \tup$  (cf. part II.1 of our theorem).


Moreover, the only non-trivial (cf. Remark \ref{rdual}(2) below)
 extension statement of this sort known to the author 
 is Theorem 6.2 of \cite{schnurdb} which relies on certain "geometric" assumptions on the initial decomposition.

  3. Recall also that  Corollary 1.12 of \cite{orl6} treats decompositions 
   of $\tub$  whose components are admissible in the sense of Definition \ref{dadm}(\ref{iadm}). This  additional restriction allows to apply arguments (in the proof of Proposition 1.10 of ibid.) that are rather similar to our ones (but avoid "auxiliary" categories) to obtain that decompositions of this sort restrict to $\tux$.
  It is worth noting that Proposition 1.10 and   Corollary 1.12 of ibid. extend to our $R$-linear setting without any difficulty.

 Note also that the arguments of the current paper were not inspired by ibid.; cf. \S\ref{swt} below for more detail on this matter. 
 \end{rema}


Let us  now 
 consider certain support subcategories.

\begin{coro}\label{csub}
Let $\bu$ be an exact
  abelian subcategory of $R-\modd$; adopt the notation and assumptions of  Theorem \ref{tcond}(I.3). 

1. Then for any decomposition $D$ of $\tux$ the couple $D^{\perp}_{\tu_{\bu^\z}}$ gives a decomposition of the corresponding category $\tu_{\bu^\z}$, and this decomposition restricts to $\tu_{\au^s}\cap \tu_{\bu^\z}$ for any $s$ as in our theorem. 

2. Assume that $\bu$ consists of $R$-modules supported on $T$, where $T$ is a union of closed subsets of $S=\spe R$ (see 
 \cite[Tags 00L1, 01AT]{stacksd}).
Then    $\tu_{\bu^\z}$ consists of all those 
   objects of $\dux$ 
  the sections of whose cohomology sheaves (note that those are $R$-modules) 
   are supported on $T$.
\end{coro}
\begin{proof}
1. 
Theorem \ref{tcond}(I.3) easily implies that both $\bu^\z$ and all $\bu^\z\cap \au^s$ satisfy the assumptions on $\au$ in Theorem \ref{tamain}(II). Now, the latter theorem yields the result immediately.

2. By the definition of support, an object $C$ of $\dux$ belongs to $\tu_{\bu^\z}$  if and only if for any object $M$ of $\tux$ and any scheme point $s_0\in S\setminus T$ we have $\dux(M,C)\otimes_R R_{s_0}=\ns$; here $R_{s_0}$ is the localization of $R$ at $s_0$. Next, 
 for the dual object $M^{\bigvee}$ we have  $\dux(M,C)\cong H^0(X,M^{\bigvee}\otimes C)$; see \cite[Tag 08JJ]{stacksd}. 
Since the ring $R_{s_0}$ is a flat $R$-module, applying the associativity of $-\otimes -$ we deduce  $$\dux(M,C)\otimes_R R_{s_0}\cong \dux(M,C\otimes_R R_{s_0}).$$ 
 Combining Theorem \ref{tcond}(1) with Lemma \ref{lbsn}(3) we obtain that $C$ belongs to $\tu_{\bu^\z}$ if and only if $C\otimes_R R_{s_0}=0$ for all ${s_0}\in S\setminus T$.
 
 Moreover, the flatness of $R_{s_0}$ implies that for any  $n\in \z$ and $U\subset X$ we have $H^n(C\otimes_R R_{s_0})(U)\cong H^n(C)(U)\otimes_R R_{s_0}$. Since a complex of sheaves is acyclic if and only if all the sections of its cohomology sheaves are trivial, we conclude that $C\otimes_R R_{s_0}=0$  if and only if the sections of  all the  sheaves   $H^n(C)$  are supported on $T$.
\end{proof}

\begin{rema}\label{rngeom}
1. Clearly, all the categories  $\tu_{\au^s}\cap \tu_{\bu^\z}$ as in 
  Corollary \ref{csub}(1) depend on the category $\bu\cap R-\mmodd$ only. Now, all intersections of this sort 
 consist of finitely generated $R$-modules supported at some $T$ as in Corollary \ref{csub}(2); see Theorem A of \cite{taka} (along with Definition 2.3(1) of ibid.).


2. The 
 functoriality of the decomposition triangles (\ref{edec}) provided by Proposition \ref{pdec}(1) implies that $r\id_L=0=r\id_R$ whenever $r\in R$ and $r\id_M=0$. 
 Possibly, this observation can be used to obtain some  result related to Corollary \ref{csub}(2). 
\end{rema}

\subsection{
The proof Theorem \ref{tnee2}(II.1) and 
  some duality arguments}\label{sunb}


\begin{proof}
  Denote the projection $X\to S=\spe R$ by $p$.

 Recall that an object $N$ of $\dux$ belongs to $\tup$ if and only if all its 
  cohomology sheaves $H^i(N)$ are coherent. Moreover, $N$ belongs
 to $\dmu^-$ (resp. $\dmu^+$) whenever 
 we also have $H^i(N)=0$ for $i\gg 0$ (resp. $i\ll 0$).
Firstly we discuss the following easy part of Theorem \ref{tnee2}(II.1): for any $M\in \obj \tux$ and $N\in \obj \tup$
 we have $\dux(M,N)\in R-\mmodd$,   and that  $\dux(M[i],N)=\ns$ whenever $N\in \obj \dmu^-$ (resp.  $N\in \obj \dmu^+$) and $i$ is small (resp. large) enough.
 Recall that   $M$ is dualizable, its dual is perfect as well, and perfect complexes have finite Tor-amplitude.
 Hence it suffices to note that the functor $Rp_*:\dux\to D(R)$ sends $\tup$ into $D(R-\mmodd)$ and has finite cohomological amplitude. 
 
 Now we verify the converse implications. We will ignore the case of $\dmu^-$ for the reasons described in the proof of Theorem \ref{tcond}(I.3); yet note that the corresponding version of our argument works 
  without any difficulty.

 We argue similarly to the proof of \cite[Theorem A.1]{bvdb}; recall  
 that $X$ is closed subscheme of the projectivization $Y$ of a vector bundle over $S$. 

 Let us reduce the latter statement to the case $X=Y$. For any $M\in D_{perf}(Y)$ we have $\dux(Li^*M,N)\cong D(\operatorname{Qcoh}(Y)) (M, i_*N)$, where $i$ is the embedding $X\to Y$.
Since $Li^*M\in\tux$, the functor represented by the object $i_*N$ fulfils the corresponding 
  assumptions, and it remains to note that 
   $N$ belongs to $\tup$ (resp. $\dmu^+$) if and only if $i_*N$ belongs to $D_{coh}(\operatorname{Qcoh}(Y))$ (resp. to $D^+_{coh}(\operatorname{Qcoh}(Y))$); see  \cite[Tags 01QY, 087T]{stacksd} (along with Remark \ref{rstacks}). 
  
Now we assume $X=Y$, and $X$ is of dimension $d\ge 0$ over $S$.
 We apply Theorem 6.7 of \cite{schnurdb}. It gives fully faithful functors $\Phi_j:D(R)\to \dux;\ F\mapsto p^*F(j)$, for $j\in \z$; here we identify $D(\operatorname{Qcoh}(S))$ with $D(R)$. Moreover, it gives a "multiple semi-orthogonal decomposition" of $\dux$ into the essential images $\imm \Phi_j$ for $0\le j\le d$; see Definition \ref{dmult}(2) below (or Definition 5.3 of ibid.). 

Let us prove by 
 induction in $m,\ -1\le m\le d$, that $N$ belongs to $ \tup$ (resp. to $\dmu^+$)  if we assume in addition that $N$ belongs to the extension-closure of $\cup_{0\le j\le m}\imm \obj\Phi_j$; we will write 
  $\dux{}_{\le m}$ for the corresponding full triangulated subcategory of $\dux$ (see Lemma \ref{lbsn}(1) or Proposition \ref{pmult}(1) below). 
  This statement is vacuous if $m=-1$.
 
 Suppose that the inductive assertion is fulfilled for $m=m_0-1$ (where $0\le m_0\le d$) and $N\in \obj \dux{}_{\le m_0}$.
  Now the subcategories $\imm \Phi_{m_0}$ and  $\dux{}_{\le m_0-1}$ give a semi-orthogonal decomposition of  $ \dux{}_{\le m_0}$; see 
    Proposition \ref{pmult}(1) below or Definition 2.2 of \cite{kuzbch}. 
   Hence there exists a distinguished triangle 
   \begin{equation}\label{edn} N'\to N\to N''\to N'[1]\end{equation}
    with $N'\in \imm  \obj \Phi_{m_0}$ and $N''\in \obj  \dux{}_{\le m_0-1}$, and $N'\cong  \Phi_{m_0}\circ \Phi_{m_0}^! (N)$; here  $\Phi_{m_0}^!$ is the right adjoint to the functor $\Phi_{m_0}:D(R)\to  \dux{}_{\le m_0}$. 
      Now, the cohomology of the complex $\Phi_{m_0}^! (N)$ is given by  $\dux(p^*R(m_0),N[i])$ for $i\in \z$ (here $R$ is the tensor unit object of $D(R)\cong D(\operatorname{Qcoh}(S))$). Hence $\Phi_{m_0}^! (N)$ belongs to  $ D_{coh}(\operatorname{Qcoh}(S))\subset  D(R-\modd)$ (resp. to 
     $D^+_{coh}(\operatorname{Qcoh}(S))$); thus $\Phi_{m_0}\circ \Phi_{m_0}^! (N)$ belongs to $ \tup$ (resp. to 
       $\dmu^+$). Moreover, 
 applying (\ref{edn}) to functors corepresented by objects of $\tux$ we obtain that $\dux(M,N'')$ belongs to $R-\mmodd$ for any $M\in \obj \tux$ (and 
  we also have  $\dux(M[-i],N'')=\ns$ for $i\ll 0$ and the
  $\dmu^+$-version of the argument). Applying the inductive assumption we deduce that $N''$ is an object of $ \tup$ (resp.   $\dmu^+$) as well; hence the same is valid for $N$ itself. 
 
 Lastly, the category $\dux{}_{\le d}$ 
  equals $\dux$; see Proposition \ref{pmult}(1) below  or combine Definition 5.3 of \cite{schnurdb} with Lemma \ref{lbsn}(1).
    \end{proof}

 Now let us 
 pass to Grothendieck duality arguments; see Definition \ref{dadm}(\ref{icoh}) for the notation.

\begin{pr}\label{pdual}

Assume that  $X$  {\it admits a dualizing complex} in the sense of \cite[
Tag 0A87]{stacksd}. 

1.  Then an exact Grothendieck duality functor $D_X: \tup\to \tup{}\opp$ is defined 
 (uniquely up to an equivalence).
 The functor  $D_X\opp\circ D_X$ is isomorphic to the identity; respectively, $D_X$ is an equivalence (and an {\it involution}).

2. $D_X$ switches $\dmu^-$ and $\dmu^+$ and fixes $\tub$. 

3. If  $Y'$ is a scheme of 
  finite type 
  over a {\it Gorenstein} scheme  $Y$ 
  (see 
  \cite[Tag 0AWW]{stacksd}) of finite Krull dimension then $Y'$   admits a dualizing complex.

Moreover, if  $X$ is  Gorenstein of finite Krull dimension 
 then  $D_X$ also restricts to an equivalence $\tux\to \tux\opp$. In particular, this is the case if $X$ is regular of finite Krull dimension.

4. If $D_0=(\lo_0,\ro_0)$ is a semi-orthogonal decomposition of a triangulated subcategory $\tu_0$ of $\tup$ then the couple $D_X(D_0)=(D_X(\ro_0),D_X(\lo_0))$  is  a semi-orthogonal decomposition of the subcategory $D_X\opp(\tu_0\opp)\subset \tup$.
\end{pr}
\begin{proof}
All statements in assertions 1--3 
  easily follow from the properties of Grothendieck  duality listed in 
   \cite[Tags 0AU3, 0DWG, 0BFQ part 2]{stacksd} 
along with \cite{hares}; see 
%
the Sufficient condition 2 in 
\S V.10 
  of ibid.

Assertion 4 is an easy consequence of our definitions; recall that $D_X$ is fully faithful and  essentially surjective. 
\end{proof}



\begin{coro}\label{cdual}
 Assume that $X$   admits a dualizing complex and $E$ 
  is a semi-orthogonal decomposition  
 of $\tub$.

I. Assume that $X$ is regular of finite Krull dimension. 

Then 
$\tux=\tub$, and 
 there exist a unique semi-orthogonal decomposition 
  $\perpp E$ of $\tux$ such that $E=(\perpp E)^{\perp}_{\tux}$. 
   Moreover, $\perpp E=D_X(D_X(E)^{\perp}_{\tux})$. 


II. Suppose 
 that $X$ is either  regular of finite Krull dimension or  satisfies Assumption \ref{assa}.
 
1. Then 
 %
  $E$ extends (see  Definition \ref{dsort}(3))  to 
 a decomposition of $\dmu^+$ of the following form: 
    $E^+=D_X( (D_X(E))^{\coprod}_{\dmu^-})$. 

2. Assume in addition that $X$ is projective over $S$. 
Then  the 
decomposition 
  $E^{\coprod}_{\tup}$  of $\tup$ (see  Theorem \ref{tone}(II.1)) 
  equals  $E^u=D_X(D_X(E)^{\coprod}_{\tup})$.  

Consequently, the aforementioned $E^+$ 
  equals $E^{\coprod}_{\dmu^+}$. 

III. Assume 
 that $X$ is a Gorenstein scheme (cf. Proposition \ref{pdual}(3)) of finite Krull dimension and $D=(\lo,\ro)$ is a semi-orthogonal decomposition of $\tux$. 

Then ${}^\perp_{\tub}D=({}^\perp_{\tub}\lo, {}^\perp_{\tub}\ro)$ is a semi-orthogonal decomposition  of $\tub$ and 
  ${}^\perp_{\dmu^+}D$ is a semi-orthogonal decomposition  of $\dmu^+$. Moreover, if $X$ is projective over $S$ 
    then  
   ${}^\perp_{\tup}D$  is a semi-orthogonal decomposition  of $\tup$.
\end{coro}
\begin{proof}
I. 
It is well known that $\tux=\tub$ in this case. 
Hence $D_X(D_X(E)^{\perp}_{\tux})$ is a decomposition of  $\tux$;  
 see Theorem \ref{tone}(I.1) and Proposition \ref{pdual}(2,4). Since $D_X$ gives an equivalence $\tux\to \tux \opp $,  $(D_X(D_X(E)^{\perp}_{\tux}))^{\perp}_{\tux}=E$ indeed; see  Proposition \ref{port}(II.1) and the proof of Theorem \ref{tone}(II.1). 

II.1. 
 According to Proposition \ref{pdual}(1,2,4), $D_X(E)$ is a semi-orthogonal decomposition of $\tub$ as well.  
 By Theorem \ref{tone}(II.1),  
  $(D_X(E))^{\coprod}_{\dmu^-}$  
   is a decomposition of $\dmu^-$ 
   that restricts to $D_X(E)$ on $\tub$. 
  Applying $D_X$ once again we obtain that 
   $E^+$  is a decomposition of $\dmu^+$ that restricts to $E$ on $\tub$. 

2. 
We similarly obtain that 
  $E^u$ 
   is a decomposition of $\tup$ that extends both  $E^+$ and  $E=(\lo_E,\ro_E)$.
  
Next, 
for $D={}^{\perp}_{\tux} (D_X(E))$ we have $D^{\perp}_{\tub}=D_X(E)$ and $D^{\perp}_{\tup}=(D_X(E))^{\coprod}_{\tup}$; see Theorem \ref{tone}(II.1). Consequently, if $D=(\lo,\ro)$ then 
$$E^u= ({}^{\perp}_{\tup}D_X(\ro),{}^{\perp}_{\tup}D_X(\lo)).$$ 
Since the decomposition $E^u=(\lo^u,\ro^u)$ 
extends $E$ and representable functors convert coproducts into products, 
  it follows that $\lo_E{}^{\coprod}_{\tup}\subset \lo^u$ and $\ro_E{}^{\coprod}_{\tup}\subset \ro^u$.
  Hence comparing 
    $E^u$ with %
    the decomposition $E^{\coprod}_{\tup}$ of $\tup$ we obtain that $E^{\coprod}_{\tup}=E^u$ indeed; see  
     Theorem \ref{tone}(II.1) and Proposition \ref{port}(II.1). 
       
  Since $E^+=E^u\cap \dmu^+$, we conclude that $E^+$ 
  equals $E^{\coprod}_{\dmu^+}$.

III. Proposition \ref{pdual}(3) allows us to deduce all these statements from Theorem \ref{tone}(I.1) easily. 
\end{proof}

\begin{rema}\label{rdual}
1. Note that people are usually interested in schemes that are of finite type over regular ones (say, over spectra of fields or Dedekind domains) only. In this case $X$ admits a dualizing complex automatically; see Proposition \ref{pdual}(3).

2. Certainly, the case where the scheme $X$ is regular itself
is quite important. Note however that in this case 
 the category $\tux=\tub$ is {\it $R$-saturated}; see Definition 4.1.1 of \cite{bvttr} that originates from Definition 2.5 of \cite{bondkaprserre}. In this case the existence of the $\tub$-adjacent semi-orthogonal decomposition (see Theorem \ref{tone}(I.1)) 
  can also be proved similarly to the rather easy Proposition 2.6 of ibid. 

3. Clearly, 
 Grothendieck duality arguments  
  can also be 
    combined with  the aforementioned Neeman's 
 Corollary 0.5 of \cite{neesat} and Theorem 0.2 of \cite{neetc} to yield certain duals of these statements. 
\end{rema}


\section{Supplements and remarks}\label{scomm}

In \S\ref{snee} we prove an abstract Theorem \ref{tison} closely related to \cite{neesat}. We use it to prove Theorem \ref{tnee2}(I); 
 next we prove part 3 of that theorem. We also 
  prove that the aforementioned Theorem 0.2 of \cite{neetc} gives certain "almost decompositions" of $\dmu^-$, and compare two methods for constructing semi-orthogonal decompositions and adjoint functors.

In \S\ref{smult} we describe the "multiple" versions of definitions and main properties of semi-orthogonal decompositions. They imply the corresponding generalizations of our central results.

In \S\ref{swt} we discuss the relation of our arguments to (adjacent) weight structures and $t$-structures.


\subsection{
 More statements related to Neeman's results}\label{snee}

Now we prove a general theorem that yields Theorem \ref{tnee2}(I). 
We use some  definitions and notation from \cite{neesat}; our arguments are also related to ibid. 

\begin{theo}\label{tison}
Assume that $\tu$ is compactly generated by its subcategory $\tu_c$, and $F$ is an object of $\tu$. 
Denote the corresponding Yoneda functor $$\tu\to\adfu(\tu_c\opp,\ab)$$ 
  (cf. Theorem \ref{tnee2})  by $\yu$.

I.1. 
Assume that $F$ is {\it $\tu_c$-approximable} in the following sense:
there exists an (infinite)  chain of morphisms $E_0\to E_1\to \dots$, and $E_i\in \obj \tu_c$ are equipped with compatible morphisms 
$s_i:E_i\to  F$ that yield an isomorphism $\inli \yu(E_i)\cong \yu(F)$. Then for any $G\in \obj \tu$ any $\adfu(\tu_c\opp,\ab)$-morphism $\yu(F)\to \yu(G)$ is induced by some  morphism $F\to G$. 
 
 Consequently, the restriction of the functor $\yu$ to the subcategory   $\tu_a\subset \tu$ 
  of  $\tu_c$-approximable objects  is full, and if $\yu(G)\cong \yu(F)$ for $F\in \obj \tu_a$ then $G\cong F$.

2. Suppose that  there exists a chain of morphisms $F_0'\to F'_1\to \dots$, and $F'_i$ are equipped with compatible morphisms $t_i$ into $F$ that yield an isomorphism $\inli \yu(F'_i)\cong \yu(F)$. Moreover, assume that there exist morphisms $c_i:E_i'\to F_i'$ such that   $E'_i\in \obj \tu_c$ and for any $T\in \obj \tu_c$ there exists $N_T\ge 0$ such that $\{T\}\perp\{\co(c_i)\}$ for all $i>N_T$. Then one can choose a subsequence $E_i$ of $E'_i$ along with some connecting morphisms between them and compatible maps $E_i\to F$ such that $\inli \yu(E_i)\cong \yu(F)$ (as in assertion I.1).

II. Assume that $\tu$ is {\it generated} by a single $G\in  \obj \tu_c$, that is, $\{G[i],\ i\in\z\}^\perp_{\tu}=\ns$. 
Then $F$ is approximable 
whenever 
 one of the following assumptions is fulfilled.

1. There exist  $c_i:E_i'\to F_i'$ and $t_i$ as in assertion I.2 such that $\{G[i],\ -j\le i\le j\in\z\}\perp \{\co(c_j),\co(t_j)\}$ for all $j\ge 0$. 

2. There exist a $t$-structure on $\tu$ and $N\in \z$ such that $G\in \tu^{\le N}$ and  $ \{G\}\perp \tu^{\le -N}$ (see Definition 1.3.1 of \cite{bbd}), and for any $i\ge 0$  there exists a morphism 
    $c'_i: \emu_i\to F^{\le i}$ such that $\emu_i\in \obj\tu_c$ and $\co(c'_i)\in \tu^{\le -i}$. 

III. One can take  $\tu=\dux$, $\tu_c=\tux$, and $G$ to be any compact generator of $\dux$ (see Example 3.4 of \cite{neesat}) 
  in assertion II.
  Moreover, all the assumptions of assertion II.2 are fulfilled whenever $t$ is the canonical $t$-structure on $\dmu$ and $F\in \obj \tup$. 
\end{theo}
\begin{proof}
I.1. The proof relies on rather easy and well-known properties of "triangulated" countable homotopy colimits. We will not define these colimits and only recall the facts we need.

So, there exists an object $F'=\hinli E_i$ along with compatible morphisms $s_i':E_i\to F'$. Next, for any object $T$ of $\tu$ 
 the corresponding homomorphism $\tu(F',T)\to \prli \tu(E_i,T)$ is surjective, and if $T_c\in \obj \tu_c$ then we obtain an isomorphism $\inli \tu(T_c,E_i)\to \tu(T_c,F')$; see Lemma 2.1.3(2--4) of \cite{bpws}. 
 
 Now we argue somewhat similarly to Lemma 5.8 of \cite{neesat}. Taking $T=F$  and the sequence $(s_i)\in \prli \tu(E_i,T)$ in the first property of $F'$ we obtain the existence of  $s:F'\to F$ such that $s\circ s_i'=s_i$ for each $i\ge 0$. Applying our assumptions on $F$ along with the second property of $F'$ we deduce that the homomorphism $\tu(T_c,s)$ is bijective for any $T_c\in \obj \tu_c$.  
 Consequently, $\tu_c\perp\{\co(s)\}$; hence Lemma \ref{lbsn}(3) implies 
   $\co(s)=0$. 
 
 Thus $F\cong F'$. Taking $G=T$ in the 
  first property of $F'$ we obtain the surjectivity statement in question. Clearly, it implies that  the restriction of $\yu$ to   $\tu_a$
 is full.
 
 Lastly, note that our assumptions on $F$ depend on $\yu(F)$ only. Thus if $\yu(G)\cong \yu(F)$ for $F\in \obj \tu_a$  then $G$ belongs to $\tu_a$ as well. Applying the fullness 
statement that we had just proved we obtain that it remains to prove the following: if $\yu(e)$ is invertible for a $\tu_a$-endomorphism $e$ then $e$ is invertible itself. Now, we obtain $\tu_c\perp\{\co(e)\}$ in this case; once again, it follows that $\co(e)=0$.

2. Let us prove that for any $j\ge 0$ there exists $l>j$ along with a morphism $E'_j\to E'_l$ such that the square
 $$\begin{CD}
 E'_j@>{}>>E'_l\\
@VV{c_j}V@VV{c_l}V \\
F'_j@>{}>>F'_l
\end{CD}$$
is commutative. 
As follows from the well-known and easy Proposition 1.1.9 of \cite{bbd}, for this purpose it suffices to choose $l$ such that $\{E'_j\}\perp\{\co(c_l)\}$. Thus we can take any $l>\max(N_{{E'_j}},j)$.


Next, we apply this statement repetitively starting from $n_0=0$ to obtain an infinite commutative diagram 
$$\begin{CD}
 E'_{n_0}@>{}>>E'_{n_1}@>{}>>E'_{n_2}@>{}>>E'_{n_3}@>{}>>\dots\\
@VV{c_{n_0}}V@VV{c_{n_1}}V@VV{c_{n_2}}V@VV{c_{n_2}}V@VV{}V \\
F'_{n_0}@>{}>>F'_{n_1}@>{}>>F'_{n_2}@>{}>>F'_{n_3}@>{}>>\dots
\end{CD}$$
Composing the compatible morphisms $F'_{n_i}\to F$ with $c_{n_i}$ we clearly obtain compatible morphisms $E'_{n_i}\to F$. We set $E_i=E'_{n_i}$.

Lastly, for any object $T$ of $\tu_c$ we 
 have $\inli_{i\ge 0} \tu(T,E_i) \cong \inli_{i>N_{T\bigoplus T[1]}} \tu(T,E_i)\cong 
 \inli_{i\ge 0} 
   \tu(T,F'_{n_i})\cong \tu(T,F)$. 

II.1. 
It clearly suffices to note that any object $T$ of $\tu_c$ is a direct summand of $T'$ such that $T'$ belongs to the extension-closure of $\{G[i]:\ -N\le i\le N\}$ for some $N>0$; see Example 0.13 and Remark 0.15 of \cite{neesat} or Proposition 4.4.1 of \cite{neebook}. 

2. Recall that  $\tu^{\le s}=\tu^{\le 0}[-s]$, $\co(F^{\le s}\to F)\in \tu^{\ge s+1}$ and $\tu^{\le s}\perp  \tu^{\ge s+1}$ for any $s\in \z$; see Definition 1.3.1 of \cite{bbd} (once again).  
 It easily follows that 
  $c_i=c'_{i+N+1}$ along with the canonical morphisms $t_i:F^{\le N+i+1}\to F$ 
  fulfil the assumptions of the previous assertion.

III. All the statements in question except the $\tu_c$-approximability of $F$ are 
 provided by Example 3.4 of \cite{neesat}.
Next, an object $L$ of $\dux$ belongs to $\tup$ if and only if all of its canonical $t$-truncations $L^{\le i}$ belong to $\dmu^-$. Thus is remains to note that $\dmu^-$ in this case equals the corresponding  subcategory $\tu^-_c$ of $\tu$;  see  Definition 0.16 of ibid. 
\end{proof}

\begin{rema}\label{rison}
1. Combining   parts I.1 and III  of our theorem we immediately obtain Theorem \ref{tnee2}(I). Moreover, this theorem also extends  
to the case where $X$ (is proper but) is not necessarily projective over $S$. 

2. 
The arguments used in \cite{neesat} to establish the fullness statement similar to Theorem \ref{tison}(I.1) require some additional assumptions (see Lemma 7.5 
  of ibid.). This restricts their "geometric" applications to the category $\dmu^-$ (instead of $\tup$ in our Theorem \ref{tison}(III)).  
Note however that the essential uniqueness for the objects that $\dux$-represent  
 those functors that correspond to $\tu_{\au^-}$ (see Theorem \ref{tcond}(I.3) or  \ref{tnee2}(II.1)) can be easily "extracted" from ibid.; cf. Lemma 5.8 of   \cite{neesat}.
 
 3. One may say that approximations used in ibid. come from certain "truncations from the left", whereas the approximations in Theorem \ref{tison}(III) come from some "two-sided truncations".
\end{rema}

To prove Theorem \ref{tnee2}(II.2)
we need a simple lemma.

\begin{lem}\label{lrlin}
 For any $R$-linear (additive) category $\bu$ the category $\adfur(\bu,R-\modd)$ is equivalent to $ \adfu(\bu,\ab)$.
 \end{lem}
 \begin{proof} Any additive functor $F:\bu\to \ab$ naturally becomes an $R$-linear one if we define the multiplication by $r\in R$ on $F(B)$ for $B\in \obj \bu$ by means of composing with $F(r\id_{\bu})$. 
\end{proof}

\begin{proof}[The proof of Theorem \ref{tnee2}(II.2)] 

 First assume that $R$ is countable.

 Then for any $Y$ that is of finite type over $S$ (that equals the spectrum of a countable ring) the category 
  $D_{perf}(Y)$  is {\it countable}, that is, the set of isomorphism classes of objects and all its morphism sets are  countable. 
 Indeed, this statement is trivial if $Y$ is affine, and the general case 
  can be reduced to this one; see also \cite[Tag 0G0W]{stacksd}. 
 Thus we can apply Theorem 5.1 of \cite{neecb} to obtain that all homological functors $\tux\opp\to \ab$ are represented by objects of $\dux$. 
  Thus all homological functors $\tux\opp\to R-\modd$ are representable as well; see Lemma \ref{lrlin}.
  Since  the values of $H$  belong to $R-\mmodd$ and $X$ is projective over $S$, $H$ is represented by an object $N$ of $\dux$ according to Theorem \ref{tnee2}(I).

  Next assume that $R$ is self-injective. Similarly to the proof of \cite[Theorem A.1]{bvdb},  we apply a double duality argument. 
    The idea is to extend a homological functor $H:\tux\opp\to R-\mmodd$ to a "nice" functor $\dux\opp\to   R-\modd$.
   
  Take the functor $\hat{H}:\tux\to R-\mmodd,\ M\mapsto \homm_R(H(M),R)$. Since $R$  is an injective $R$-module, $\hat{H}$ is homological. Next, it extends to a homological functor $\hat{H}':\dux\to R-\modd$ that respects coproducts; see Proposition 2.3 of \cite{krause}. Now we take $H':\dux\opp\to R-\modd,\  M\mapsto \homm_R(\hat{H}'(M),R)$.
  This functor is clearly cohomological and respects products. Consequently, $H'$ is representable; 
  see the well-known Theorem 8.3.3 
   of \cite{neebook} along with Lemmata \ref{lbsn}(4) and \ref{lrlin}. 
  
  It remains to prove that $H'$ restricts to $H$ on $\tux$. We note that $R$ is a  {\it quasi-Frobenius ring}; see Theorem 15.1 of \cite{lamod}. Hence the ("double duality") statement is question is provided by Theorem 15.11 of ibid.
\end{proof}


\begin{rema}\label{rnee212}
1. Since $R$ is commutative, it is also a {\it Frobenius ring} whenever it is self-injective. Moreover, rings of this sort can be described more or less explicitly; see  Theorem 15.27  of \cite{lamod}.

Note however that the most important ("from the algebraic geometry point of view") 
Frobenius rings are fields; this is the only case mentioned in \cite{bvdb}.  

2.  Once again, Corollary 0.5 of \cite{neesat} 
 suggests that the (additional) assumptions on $R$ in Theorem \ref{tnee2}(II.2) are not necessary; it also probably suffices to assume that $X$ is proper over $S$. 
\end{rema}

Now we deduce one more consequence from Theorem 0.2 of \cite{neetc}.
We argue somewhat similarly to Proposition \ref{ptu}.

\begin{pr}\label{pnee3}
  Assume that $X$ satisfies Assumption \ref{assa} 
  and let $D=(\lo,\ro)$ be a semi-orthogonal decomposition of $\tub$. 
 
Then  for any $M\in \obj \dmu^-$ there exists a distinguished triangle $$ L\to M\to R\to L[1]$$ with $L\in {}^\perp_{\dmu^-}\lo$ and $R\in {}^\perp_{\dmu^-}\ro$.

\end{pr}
\begin{proof}
We fix $M$. Let $H_L:\tub\to R-\modd$ be the functor $N\mapsto \dmu^-(M, R_D(N))$ (see Proposition \ref{pdec}(1));  
 it is obviously homological. 
Now recall that the Yoneda-type functor $\dmu^-{}\opp\to \adfur(\tub,R-\modd)$ is full, and its essential image 
 consists of all homological functors $\tub\to R-\mmodd$; see Theorem 0.2 of \cite{neetc}. Hence all values of $H_L$ belong to $R-\mmodd$, and applying loc. cit. once again we obtain that $H_L$ is isomorphic to  the functor $N\mapsto \dmu^-(L, N)$ for some $L\in \obj \dmu^-$. Since $R_D$ annihilates $\lo$, $L$ belongs to $  {}^\perp_{\dmu^-}\lo$.

Next, set $H_M:\tub\to R-\modd$ to be the functor $N\mapsto \dmu^-(M, N)$;
take $\Phi: H_L\to H_M$ to be the transformation induced by the morphisms $N\to  R_D(N)$ in (\ref{edec}). Then loc. cit. also implies that this transformation comes from a (possibly, non-unique) morphism $f: L\to M$.

 Lastly, if $N\in \ro$ then $N\cong  R_D(N)$; hence in the exact sequence $$\dmu^-(M[1],N)\to \dmu^-(L[1],N)\to \dmu^-(\co(f),N) \to \dmu^-(M,N)\to \dmu^-(L,N)$$ the first and the last maps are isomorphisms, and we obtain  $\co(f)\in {}^\perp_{\dmu^-}\ro$.
\end{proof}

\begin{rema}\label{rimpr}
1. Clearly, this argument can be axiomatized  similarly to Theorem \ref{tamain}(II). 
In particular, one can 
 combine it with  Theorem \ref{tnee2}(II.2) (see Remark \ref{rnee2}(3)) to obtain a weaker version of Theorem \ref{tone}(I.1).

Moreover, 
 note that the corresponding  Yoneda-type functor $\yu_{\tux\opp}: \tux\opp\to \adfur(\tub,R-\modd)$ is fully faithful; this statement (that is contained in Theorem 0.2 of \cite{neetc} as well) is obvious since $\tux\subset \tub$. Loc. cit. also says that the image of $\yu_{\tux\opp}$ consists of all {\it finite homological functors} (cf. 
   Theorem \ref{tcond}(II.2)). Consequently, the corresponding version of the argument in the proof of Proposition \ref{pnee3} yields an alternative proof of Theorem \ref{tone}(II.1) in the case where Assumption \ref{assa} is fulfilled, and  it does not require any mock projectives (cf. Theorem \ref{tcond}(II)). 
  
  However, it appears that it makes sense to construct adjacent decompositions using Theorem \ref{tamain}(II.2) (possibly combining it with Proposition \ref{ptu}) 
   since this method  requires less knowledge on the relation between $\tu$ and $\tu'$; cf. Theorem \ref{tnee2}(II.1). 
     We have to pay the price of specifying certain $\dmu$ 
  which can be "not that interesting" (cf. Theorem \ref{tcond}(II) and Remark \ref{rdec1}(1)). 
  
  However, some $\dmu$ as desired is well known to exist "in all reasonable cases". 
  Similarly, a functor $\tu_{\au}\to D$ as in Corollary \ref{cadj} "usually" can be extended to an exact functor $\dmu\to \dmu'$ that respects coproducts. 
  For this reason, the author suspects that Corollary \ref{cadj} is somewhat more useful than the corresponding Corollary 0.4 of \cite{neesat}; cf. Remark \ref{rcadj}.
 Note also that the proof of that corollary is much easier than that of loc. cit.
 
 2. Moreover, Theorem 0.3 of \cite{neesat} and other results of Neeman may help in constructing "interesting" couples $(\tu,\tu')$ that satisfy the assumptions of  Proposition \ref{ptu}.
  

%
3. Proposition 2.5.4(1) of \cite{bger} easily implies that for $R=\co(f)$ the functor $N\mapsto \tu'(M, R)$ is isomorphic to  $N\mapsto \tu'(M, R_D(N))$; cf. Remark \ref{radj}. Moreover, we obtain a complete identification of the transformations between the corresponding Yoneda-type functor from part 2 of that proposition.
\end{rema}

\subsection{On "multiple" semi-orthogonal decompositions}\label{smult}

Now we generalize Definition \ref{dsort}(1,3).

\begin{defi}\label{dmult}
Let $n\ge 1$ and assume that $\tu_i,\ 0\le i\le n,$ are 
 (strictly full) 
  triangulated subcategories of $\tu$. 

1. Then for any $j,\ -1\le j\le n$ we will write $\tu_{\le j}$ (resp. $\tu_{\ge n-j}$) for the smallest (strictly full) triangulated subcategory of $\tu$ that contains $\tu_i$ for all $i\le j$ (resp. $i\ge n-j$).\footnote{Respectively, $\tu_{\le -1}=\tu_{\ge n+1}=\ns$.}

2. We will say that the family $(\tu_i)$ gives  a 
 {\it (length $n$) semi-orthogonal decomposition} of $\tu$ (or just 
  a decomposition of $\tu$) if $\tu_j\perp \tu_i$ whenever $0\le i<j\le n$, and 
 $\tu_{\le n}=\tu$.
 
 3.  Let $\tu'$ be a triangulated subcategory of $\tu$. We will say that a (semi-orthogonal) decomposition  $(\tu_i)$ of $\tu$ {\it restricts} to $\tu'$ whenever the family 
$(\tu_i) \cap{\tu'}$ (see Definition \ref{dadm}(\ref{irest})) gives a decomposition of $\tu'$.
 \end{defi}

\begin{pr}\label{pmult}

1.  If $(\tu_i)$ is a  semi-orthogonal decomposition of $\tu$ then for any $j,\ 1\le j\le n$, the couple $(
\tu_j,
 \tu_{\le j-1})$ gives a decomposition of $\tu_{\le j}$ 
 in the sense of Definition \ref{dsort}(1),  $(\tu_{\ge j}, \tu_{j-1})$ is decomposition of $\tu_{\ge j-1}$, 
 and $(\tu_{\ge j},\tu_{\le j-1})$ is a semi-orthogonal  decomposition of $\tu$.

2. The correspondence  $(\tu_i)\mapsto (\tu_{\le i-1},\ 1\le i\le n)$  (resp.   $(\tu_i)\mapsto (\tu_{\ge n+1-i},\ 1\le i\le n)$) gives a  bijection between the class of all length $n$  semi-orthogonal decompositions of $\tu$ and the class of   
 ascending chains  $(\ro_i),\ 1\le i\le n,$ of  left  (resp. $(\lo_i),\ 1\le i\le n$, of right) admissible subcategories of $\tu$.

Moreover, the inverse map is given by sending  $(\ro_i)$ into 
 $({}^{\perp}_{\ro_{i+1}}\ro_{i})$  (resp.   $(\lo_i)$ into $(\lo_{n-i}{}^{\perp}_{\lo_{n+1-i}})$) for $0\le i\le n$;
 here we expand  the chains $(\lo_i)$ and $(\ro_i)$ by setting  $\lo_0=\ns=\ro_0$ and $\lo_{n+1}=\tu=\ro_{n+1}$. 
    Furthermore, the corresponding bijection between  (length $n$ 
   ascending) "right admissible chains" and "left admissible chains" in $\tu$ is given by sending $(\lo_i)$ into $(\lo_{n+1-i}{}^{\perp}_{\ro})$.
  
3. Let $\tu'$ be a triangulated subcategory of $\tu$. Then the following conditions for a decomposition  $(\tu_i)$ of $\tu$ are equivalent. 

(a). 
 $(\tu_i)$  restricts to $\tu'$.

(b).  The smallest  triangulated subcategory of  $\tu'$ that contains all $\tu_i\cap \tu'$ is $\tu'$ itself.


(c).   $(\tu_{\ge j+1}, \tu_j)$ restricts to a semi-orthogonal decomposition of $\tu_{\ge j}\cap \tu'$ whenever $0\le j\le n-1$.

(d). $(\tu_{\ge j+1},\tu_{\le j})$  restricts to a semi-orthogonal decomposition of $\tu'$ for  $0\le j\le n-1$.

\end{pr}
\begin{proof}

1. Clearly, 
$\tu_j\perp \tu_{\le j-1}$. The existence of decompositions of the type (\ref{edec}) for all objects of $\tu_{\le j}$ easily follows from  Lemma \ref{lbsn}(1); cf. the proof of Theorem \ref{tamain}(I.1).

The statements that  $(\tu_{\ge j}, \tu_{j-1})$ is decomposition of $\tu_{\ge j-1}$ and $(\tu_{\ge j},\tu_{\le j-1})$ is a decomposition of $\tu$ are proved similarly.

2. Clearly, it suffices to study the first correspondence since the second one is essentially its categorical dual.

Applying assertion 1 along with Proposition \ref{pdec} we obtain that $\tu_{\le i-1}$ is left admissible in $\tu_{\le i}$ for $1\le i\le n$. Since  $\tu_{\le n}=\tu$, we obtain that all $\tu_{\le i}$  are left admissible in $\tu$.

Conversely, if one starts from
 for  $(\ro_i)$ as above then 
  she clearly obtains $\tu_j\perp \tu_i$ whenever $0\le i<j\le n$. 
Applying easy induction along with Proposition \ref{pdec}(2) we also obtain that  $\tu_{\le n}=\tu$ indeed.

Now let us calculate the compositions of our maps. 
If we start from $(\tu_i)$ then Proposition \ref{pdec}(1) easily implies that the corresponding orthogonals 
 equal $\tu_i$.
If we start from $(\ro_i)$  then the existence of the triangles of the type (\ref{edec})  provided by Proposition \ref{pdec}(2) yields that the corresponding smallest triangulated subcategories of $\tu$ are $\ro_i$ indeed.

Lastly, if we start from a right admissible chain $\tu_{\le i}$ that corresponds to $\tu_i$ then the corresponding left admissible chain consists of $\tu_{\ge n+1-i}=\tu_{\le n-i}{}^{\perp}_{\tu}$; here we apply assertion 1. 

3. The equivalence of conditions (a) and (b) is obvious. Next, applying assertion 1 (along with Proposition \ref{pdec}(1)) we easily obtain that (b) is equivalent to (c) and (d).
\end{proof}

Now we are able to generalize Theorem \ref{tamain}.

\begin{theo}\label{tamult}

Assume 
 that $\dmu=\tu^{\coprod}$, 
    where $\tu\subset \dmu$ is a triangulated subcategory  whose objects are $\dmu$-compact,   and $(\tu_i),\ 0\le i\le n,$ is a semi-orthogonal decomposition of $\tu$. 

I.1. 
Then the family $(\tu_i^{\coprod})$ 
 (see  Definition \ref{dadm}(\ref{icopr}))  is a semi-orthogonal decomposition of $\dmu$. 

2. Assume in addition that $\tu$ is essentially small (consequently, $\dmu$ is compactly generated by it). Then the family $(\hat \tu_i)=(\tu_{\ge n-i+1}{}^{\perp}_{\tu_{\ge n-i}^{\coprod}})$
is  a semi-orthogonal decomposition  of $\dmu$ as well; here we take $0\le i\le n$  (note that 
$\tu_{\ge n+1}=\ns$).

Moreover, the corresponding 
 ascending chain of  left (resp. right) admissible subcategories of $\dmu$ (see Proposition \ref{pmult}(2))  equals $(\tu_{\le n-i}{}^{\perp}_{\dmu})$ (resp. $(\tu_{\le i-1}^{\coprod})$); here $1\le i\le n$.



3. Assume that $\tu_0$ is a subcategory of $\dmu$ such that 
 the family $(\hat\tu_i{})$  restricts to a semi-orthogonal decomposition 
 $D_0$ on it (see Definition \ref{dmult}(3)). 

Then 
the family $(\hat\tu_i{})$  also restricts to the category $\tu_{0}^{\coprod}$ as well, and this restriction equals  $D_0^{\coprod}$ 
  (see  Definition \ref{dadm}(\ref{icopr}) once again).  


II. Assume that $R$, $\au$, and $\tu_{\au}$ are as in Theorem \ref{tamain}(II).
 
Then the family     $(\hat\tu_i)$ 
(see assertion I.2)   
 restricts to a decomposition of the category $\tu_{\au}$ (which is triangulated according to  Theorem \ref{tamain}(II.1)). 


\end{theo}
\begin{proof}
I.1. The proof of Theorem \ref{tamain}(I.1) 
 extends to this setting straightforwardly. 

2. 
Combining  
 the previous assertion with 
Proposition \ref{pmult}(1) we obtain that $(\tu_{\ge i}^{\coprod},\tu_{\le i-1}^{\coprod})$ is a semi-orthogonal decomposition of $\dmu$ whenever $1\le i\le n$. Combining Theorem \ref{tamain}(I.2) with Proposition \ref{pdec} we obtain that the subcategory $(\tu_{\ge i}^{\coprod})^{\perp}_{\dmu}=\tu_{\le i-1}^{\coprod}$ is left admissible in $\dmu$. 
Applying Proposition \ref{pmult}(2) 
 we deduce that the family $((\tu_{\ge n-i+1}^{\coprod})^{\perp}_{\tu_{\ge n-i}^{\coprod}})=(\hat\tu_i)$ gives a semi-orthogonal decomposition of $\dmu$ indeed.

Lastly, the corresponding "left admissible chain" equals $(\tu_{\le n-i}^{\coprod})^{\perp}_{\dmu}= ((\tu_{\le n-i})^{\perp}_{\dmu})$; see Proposition \ref{pmult}(2).


3. The proof of Theorem \ref{tamain}(I.1) carries over to this setting easily (as well) if one applies Proposition \ref{pmult}(3). 

II. According to Proposition \ref{pmult}(3), it suffices to verify that the corresponding decompositions $(\tu_{\le i-1}^{\coprod}, \tu_{\le i-1}{}^{\perp}_{\dmu})$ of $\dmu$ (see assertion I.2) restrict to $\dmu_{\au}$ for $1\le i\le n$. The latter statement immediately follows from Theorem \ref{tamain}(II.2).

\end{proof}

\begin{rema}\label{rmult}
One can also describe the family $(\hat\tu_i)$  as follows: $$\hat \tu_i=\{M\in \obj \dmu:\ \tu_j\perp M\ \forall j,\ 0\le j\le n,\ j\neq 
n- i\}.$$ 
\end{rema}

\begin{coro}\label{cmult}
The easily formulated "multiple decomposition" versions of Theorem \ref{tone} and  Corollary \ref{cdual} (one should just use the correspondence $(\tu_i)\mapsto (\hat\tu_i)$ instead of $D\mapsto D^{\perp}_{\dmu}$ in them) are valid. 

\end{coro}
\begin{proof}
Given Theorem \ref{tamult} and Proposition \ref{pmult}, all the arguments used for the proof of 
 Theorem \ref{tone} and  Corollary \ref{cdual}  carry over to the "multiple decomposition context" without any difficulty. 
\end{proof}

\subsection{On the relation to  weight and $t$-structures}\label{swt}


Now we recall that a semi-orthogonal decomposition couple gives both a {\it weight structure} and a {\it$t$-structure}; see Proposition 3.4(4) and Remark 3.5(2) of \cite{bvt}. The (main) difference between the latter notions and Definition \ref{dsort}(2) is that we do not require the corresponding $\lo$ and $\ro$ to be the object classes of triangulated subcategories of $\tu$. Instead, we only demand that $\lo\subset \lo[1]$ and $\ro[1]\subset \ro$ for weight structures and vice versa for $t$-structures (see  Definition 3.1 and Proposition 3.2(1,2) of ibid.). 

If one uses the so-called homological conventions for weight and structures 
(see Definitions 1.1.1 and 1.2.1, and  Remarks 1.1.3(4) and  1.2.3(3)  of \cite{bvttr})  then one 
 usually passes to the couples $w=(\tu_{w\le 0},\tu_{w\ge 0})=(\lo,\ro[-1])$ and  $t=(\tu_{t\le 0},\tu_{t\ge 0})=(\ro[1],\lo)$, respectively.\footnote{Note that $(\tu_{t\le 0},\tu_{t\ge 0})$ corresponds to the couple $(\tu^{t\ge 0},\tu^{t\le 0})$ in the "cohomological" notation that was used in \cite{bbd} and \cite{neesat}.} 
 Consequently, for any object $M$ of $\tu$ and $n\in \z$ there exists an  $n$-weight decomposition 
  triangle 
$$L^w_nM\to M\to R^w_{n+1}M\to L^w_{n}M[1]$$ with $L^w_nM\in \tu_{w\le 0}[n]$ and $R^w_{n+1}M\in \tu_{w\ge 0}[n+1]$ and 
and $n$-$t$-decomposition triangle $L^t_nM\to M\to R^t_{n-1}M\to L^t_{n}[1]$ with $L^t_nM\in \tu_{t\ge 0}[n]$ and $R^t_{n-1}\in \tu_{t\ge 0}[n-1]$; see Remarks  1.1.3(1) and 1.2.3(2) of ibid. If $w$ (resp. $t$) comes from a semi-orthogonal decomposition $D$ then Proposition \ref{pdec} 
 gives canonical isomorphisms $R_DM\to L_m^wM$ and $L_DM\to R_m^wM$  (resp. $R_DM\to L_m^tM$ and $L_DM\to R_m^tM$)     for any $m\in \z$.

\begin{rema}\label{radj}
1. Now let us recall the predecessors of our Definition \ref{dsort}(3).

The notion of (left or right) adjacent weight and $t$-structures was introduced in \cite{bws}; a weight structure on $\tu$ was said to be left adjacent to $t$ on $\tu$  if $\tu_{w\ge 0}=\tu_{t\ge 0}$ (if one uses the aforementioned homological conventions). If this is the case then $\tu_{t\le -1}=\tu_{w\ge 0}\perpp $ and $\tu_{t\ge 1}=\tu_{w\le 0}\perpp $; cf. our  Definition \ref{dsort}(3).

Next, in \cite{bger} weight and $t$-structures on (distinct) triangulated categories $\tu$ and $\tu'$ endowed with a so-called {\it duality} bi-functor $\tu\opp\times \tu'\to \ab$ were considered (moreover, one can replace $\ab$ by an arbitrary abelian category here). 
  In this setting a weight structure $w$ on $\tu$ was said to be {\it (left)  orthogonal} to a $t$-structure $t$ on $\tu'$  whenever $\Phi (X,Y)=0$ if $X\in \tu_{w\le 0}$ and $Y\in \tu'_{t \ge 1}$ or if $X\in \tu_{w\ge 0}$ and $Y\in \tu'_{t \le -1}$. A simple example of a duality is given by the corresponding restriction of the bi-functor $\dmu(-,-)$ whenever $\tu$ and $\tu'$ are triangulated subcategories of $\dmu$; see Remark 5.2.2(1) of \cite{bvttr}.

  Clearly, this orthogonality condition is fulfilled (for this choice of $\Phi$) in the setting of our  Definition \ref{dsort}(3). 
  
  2.  Proposition 2.5.4(1) of \cite{bger} in our setting (and notation) says
  \begin{equation}\label{eger}
 \begin{split}  \dmu(M,L^t_nN)\cong \imm(\dmu (R^w_{n}M, N)\to \dmu(R^w_{n-1}M, N))\\
 {\text{ and }}\dmu(M,R^t_nN)\cong \imm(\dmu ( L^w_{n+1}M)\to \dmu( L^w_nM,N)) \end{split}
  \end{equation}
   for any $n\in \z$; cf. Definition 2.1.1(1) and \S2.2 of \cite{bvttr}. 
If $w$ and $t$ come from $D$ and $D'$  in the setting of Proposition \ref{port}(I) then the connecting morphisms  $R^w_{n-1}M\to R^w_{n}M$ and $ L^w_{n}M\to L^w_{n+1}M$  in 
 (\ref{eger}) are just $\id_{R_D(M)}$ and $\id_{L_D(M)}$, respectively. This concludes the proof of  Proposition \ref{port}(I).
\end{rema}

Now we explain the relation of the current paper to (\S\S4--5 of) \cite{bvttr} whose main results inspired the current texts. In \S4 of ibid. it was demonstrated that in the settings closely related to the ones above one can construct (left or right) orthogonal $t$-structures on $\tu'$ from certain weight structures on $\tu$; these arguments are closer to that of the current paper than the "converse ones" of \cite[\S5]{bvttr} (see below).

However, the results of 
 \cite[\S4]{bvttr} suffer from two disadvantages. Firstly,  the author does not know of any "general"  methods for constructing  weight structures on $\tu$ of this sort (in contrast to the "smashing setting" treated in \S3 of ibid.; yet cf. Remark 4.2.2(5) of ibid.).  

Secondly, to compute  $\dmu(M[m],L^t_nN)$ using (\ref{eger}) for all $m\in \z$ one needs certain values of $\dmu(L^w_jM,-)$ for all $j\in \z$. For this reason one requires a certain "stabilization" of $L^w_jM$ (which is equivalent to the stabilization of  $R^w_{j+1}M$)  for $j\ll 0$ or $j\gg 0$; 
  cf. Theorem 4.1.2(I.1, II.1). Note here that in loc. cit. it was assumed that $w$ is {\it bounded} ({\it above, below}, or both; see Definition 1.2.2(7) of ibid.); this corresponds to the vanishing of certain ("candidates for"; see Remark 1.2.3(2) of ibid.) $L^w_jM$ and $R^w_jM$.  
It appears that these conditions can be weakened. Instead, it suffices to have certain functors $M\mapsto "L_{\ll 0}M"$ and $M\mapsto "R_{\gg 0}M"$ that correspond to certain semi-orthogonal decompositions; thus, our Theorem \ref{tamain} can probably be "mixed" with  Theorem 4.1.2 of ibid. Possibly, the author will study this question in more detail later. However, currently there is little hope to obtain "interesting geometric" examples for it that do not correspond to semi-orthogonal decompositions (and that are not treated in this paper, respectively).

Next, the "inverse correspondence $t\to w$" was studied in Theorem 5.3.1(II,IV) of ibid. The main disadvantage of all results of this sort  is that they require enough projectives in the abelian category $\hrt=\tu'_{t\le 0}\cap \tu'_{t\ge 0}$ (at least, in the case $\tu\subset \tu'$); certain boundedness assumptions are also needed. 
   This makes the construction of examples quite difficult (as well).







\end{document}